\documentclass[reqno,a4paper]{amsart}
\usepackage{eucal,amsfonts,amssymb,amsmath,amsthm,epsfig,mathrsfs}

\usepackage{amscd,amsxtra}
\usepackage{enumerate}
\usepackage{latexsym}

\allowdisplaybreaks

 \makeatletter \@addtoreset{equation}{section}

\makeatother \makeatletter

\newtheorem{thm}{Theorem}[section]
\newtheorem{hyp}[thm]{Hypotheses}{\rm}
{\rm}
\newtheorem{lemm}[thm]{Lemma}
\newtheorem{coro}[thm]{Corollary}
\newtheorem{prop}[thm]{Proposition}
\newtheorem{defi}[thm]{Definition}
\newtheorem{rmk}[thm]{Remark}{\rm}

\newcommand{\R}{{\mathbb R}}
\newcommand{\N}{{\mathbb N}}

\newcommand{\Rd}{\mathbb R^d}
\newcommand{\Om}{{\Omega}}

\newcommand{\A}{\mathcal{A}}

\newcommand{\bd}{\begin{defi}}
\newcommand{\ed}{\end{defi}}

\newcommand{\be}{\begin{equation}}
\newcommand{\ee}{\end{equation}}
\newcommand{\barr}{\begin{array}}
\newcommand{\earr}{\end{array}}
\newcommand{\bmn}{\begin{eqnarray}}
\newcommand{\emn}{\end{eqnarray}}
\newcommand{\bnm}{\begin{eqnarray*}}
\newcommand{\enm}{\end{eqnarray*}}
\newcommand{\bln}{\begin{subequations}}
\newcommand{\eln}{\end{subequations}}

\newcommand{\ba}{\begin{align}}
\newcommand{\ea}{\end{align}}
\newcommand{\banm}{\begin{align*}}
\newcommand{\eanm}{\end{align*}}
\newcommand{\one}{\mbox{$1\!\!\!\;\mathrm{l}$}}

\title[On the Dirichlet and Neumann evolution operators in $\R^d_+$]{On the Dirichlet and Neumann evolution operators in $\R^d_+$}
\author[L. Angiuli]{Luciana Angiuli}
\author[L. Lorenzi]{Luca Lorenzi}
\address{Dipartimento di Matematica e Informatica, Universit\`a degli Studi di Parma, Parco Area delle Scienze 53/A, I-43124 Parma, Italy.}
\email{luciana.angiuli@unipr.it}
\email{luca.lorenzi@unipr.it}

\date{}
\begin{document}

\begin{abstract}
We prove some uniform and pointwise gradient estimates for the Dirichlet and the Neumann evolution operators $G_{\mathcal{D}}(t,s)$ and $G_{\mathcal{N}}(t,s)$ associated with a class of nonautonomous elliptic operators $\A(t)$ with unbounded coefficients defined in $I\times \Rd_+$ (where $I$ is a right-halfline or $I=\R$).
We also prove the existence and the uniqueness of a tight evolution system of measures $\{\mu_t^{\mathcal{N}}\}_{t \in I}$ associated with $G_{\mathcal{N}}(t,s)$, which turns out to be sub-invariant for $G_{\mathcal{D}}(t,s)$, and we study the asymptotic behaviour of the evolution operators $G_{\mathcal{D}}(t,s)$ and $G_{\mathcal{N}}(t,s)$ in the $L^p$-spaces related to the system $\{\mu_t^{\mathcal{N}}\}_{t \in I}$.
\end{abstract}

\maketitle

\section{Introduction}

The increasing interest in Kolmogorov equations is due to their relevant role in many branches of mathematics. In particular, these equations arise in a natural way from many applications in physics.
For example in some free boundary problems in combustion theory and in the study of the Navier-Stokes equations in rotating exterior domains, simple changes of variables transform operators with bounded coefficients into operators with unbounded coefficients. Kolmogorov equations are also strongly connected to the study of many problems in dynamic population and in mathematical finance that lead to stochastic models where it is quite natural to require that the unbounded coefficients be explicitly depending on time. Whereas the theory is already well developed in the autonomous case (see e.g., \cite{BerFor04Grad, BerForLor07Poi, BerForLor07Gra, ForMetPri04Gra,feller} and the monograph \cite{BerLorbook}), in the nonautonomous case, some results have been proved
very recently and a lot of significant problems are still open.
To the best of our knowledge, all the literature in the nonautonomous setting is related to the case of the whole space $\Rd$. In such a case, many aspects of the Cauchy problem for nonautonomous parabolic equations have been studied in \cite{AngLor12OnI,AngLor10Com,AngLorLun,GeisLun08, GeisLun09, KunLorLun09Non,LorZam09Cor,LorLunZam}.

This paper represents the first step to understand and analyze nonautonomous elliptic operators (and their associated evolution operators) in unbounded domains with homogeneous boundary conditions.
Given a right halfline $I$ (possibly $I=\R$),
we consider a class of linear nonautonomous second-order uniformly elliptic operators
\begin{align*}
\A(t)=\sum_{i,j=1}^d q_{ij}(t,\cdot)D_{ij}+\sum_{i=1}^d b_i(t,\cdot) D_i-c(t,\cdot),
\end{align*}
with sufficiently smooth and possibly unbounded coefficients defined in $I\times \Rd_+$, where $\Rd_+:=\R^{d-1}\times
(0,+\infty)$. Under suitable assumptions on the coefficients of the operator $\A(t)$, for any $s\in I$ the
Cauchy-Dirichlet problem
\begin{equation}\label{NADir}
\left\{
\begin{array}{ll}
u_t(t,x)=\A(t)u(t,x), &t\in(s,+\infty),\,x\in \R^d_+\\[1mm]
u(t,x)=0,&t\in(s,+\infty),\,x\in \partial\R^d_+,\\[1mm]
u(s,x)= f(x),& x\in \R^d_+,
\end{array}
\right.
\tag{$P_{\mathcal D}$}
\end{equation}
with $f\in C_b(\Rd_+)$, and the Cauchy-Neumann problem
\begin{equation}
\label{NANeu}
\left\{
\begin{array}{ll}
u_t(t,x)=\A(t)u(t,x), &t\in(s,+\infty),\,x\in \R^d_+,\\[1.5mm]
\displaystyle\frac{\partial u}{\partial \nu}(t,x)=0,&t\in(s,+\infty),\,x\in \partial\R^d_+,\\[1.5mm]
u(s,x)= f(x),& x\in \overline{\R^d_+},
\end{array}
\right.
\tag{$P_{\mathcal N}$}
\end{equation}
with $f\in C_b(\overline{\R^d_+})$, are governed by two evolution operators:
the \emph{Dirichlet evolution operator} $\{G_{\mathcal D}(t,s): t\ge s\in I\}$ and the \emph{Neumann evolution operator} $\{G_{\mathcal N}(t,s): t\ge s\in I\}$.
Our aim consists in investigating some properties of these evolution operators.
In the first part of the paper we prove some pointwise gradient estimates satisfied by the functions $G_{\mathcal D}(t,s)f$
and $G_{\mathcal N}(t,s)f$. More precisely, for any $p>1$ we prove that there exist two positive constants $c_p$ and $C_p$ such that
\begin{equation}
|\nabla_x G_{\mathcal D}(t,s)f|^p\leq  c_pe^{C_p(t-s)}G_{\mathcal N}(t,s)(|f|^p+|\nabla f|^p),
\label{grad-1-intro}
\end{equation}
for any $t>s\in I$ and $f\in C^1_b(\overline{\Rd_+})$ which vanishes on $\partial\Rd_+$, and
\begin{equation}
|\nabla_x G_{\mathcal N}(t,s)f|^p\leq  c_pe^{C_p(t-s)}G_{\mathcal N}(t,s)(|f|^p+|\nabla f|^p),
\label{grad-2-intro}
\end{equation}
for any $t>s\in I$, any $p>1$, any $f\in C^1_b(\overline{\Rd_+})$.
Clearly, in \eqref{grad-1-intro} an estimate with $G_{\mathcal D}(t,s)$ in the right-hand side can not hold
since $G_{\mathcal D}(t,s)(|f|^p+|\nabla f|^p)$ vanishes on $\partial\Rd_+$, whereas, in general,
$\nabla_x G_{\mathcal D}(t,s)f$ does not.
Our main assumptions are a dissipativity condition on the drift $b=(b_i)_{i}$ and some growth assumptions on the spatial derivatives of the diffusion coefficients $q_{ij}$ and on the potential term $c$.
Under stronger assumptions we obtain \eqref{grad-1-intro} and \eqref{grad-2-intro} also for $p=1$.


We also prove that, for any $s\in I$, the estimate
\begin{eqnarray*}
|\nabla_x G_{\mathcal I}(t,s)f|^p\leq \tau_p\,e^{\omega_p (t-s)}(t-s)^{-\frac{p}{2}}G_{\mathcal N}(t,s)|f|^p), \qquad\;\, {\mathcal I}\in\{{\mathcal D}, {\mathcal N}\},
\end{eqnarray*}
holds in $\overline{\Rd_+}$ for any function $f\in C_b(\Rd_+)$ (resp. $f\in C_b(\overline{\Rd_+})$, if $\mathcal{J}=\mathcal{N}$), any $t\in (s,+\infty)$, any $p\in (1,+\infty)$ and some constants $\tau_p>0$, $\omega_p\in \R$.

Besides their own interest, the previous estimates represent a helpful tool both in studying of the asymptotic behaviour of
the evolution operators $G_{\mathcal{D}}(t,s)$ and $G_\mathcal{N}(t,s)$ and in establishing some summability improving results for such operators.
As already noticed in the case of the whole space (see \cite{AngLor10Com}), the usual $L^p$-spaces are not the appropriate setting where to study elliptic operators with unbounded coefficients and their associated evolution operators. On the contrary the $L^p$-spaces related to particular systems of measures, called \emph{evolution systems of measures} (see Definition \ref{def-4.1}), seem to be more apt.
Existence and uniqueness of such systems of measures have been proved in the case of the whole space, first for the Ornstein-Uhlenbeck evolution operator and, then, for more general nonautonomous elliptic operators with unbounded coefficients in \cite{GeisLun08,KunLorLun09Non}. We also quote the related papers
\cite{BDPR08,BDPRS,BogKryRoc01, BRS06}.

Here, in the case when $c\equiv 0$, we prove that there exists an evolution system of measures $\{\mu_t^{\mathcal{N}}\}_{t\in I}$ associated with the evolution operator $G_{\mathcal N}(t,s)$, which turns out to be sub-invariant for the Dirichlet evolution operator $G_{\mathcal D}(t,s)$ even if $\inf_{I\times \Rd_+}c\ge 0$. This family of measures is obtained as the weak$^*$ limit of the evolution systems of measures for the evolution operators $G^{\varepsilon}(t,s)$ in the whole of $\Rd$. Here, $G^{\varepsilon}(t,s)$ is the evolution operator associated with the uniformly elliptic operator $\A^\varepsilon(t)$, whose coefficients are defined in the whole of $I\times\Rd$ starting from the coefficients of $\A(t)$.

Moreover, under suitable assumptions, the gradient estimate \eqref{grad-2-intro} implies both that $\{\mu_t^{\mathcal N}\}_{t\in I}$ is the unique tight evolution system of measures for $G_{\mathcal N}(t,s)$ and that the operators $G_{\mathcal D}(t,s)$ and $G_{\mathcal N}(t,s)$ are bounded from $L^p(\Rd,\mu_s^{\mathcal{N}})$ into
the Sobolev space $W^{1,p}(\Rd,\mu_t^{\mathcal{N}})$ for any $t>s\in I$.

As in the case of the whole space, the unique tight evolution system of measures appears naturally in the study of the asymptotic behaviour of $G_{\mathcal{N}}(t,s)$ and $G_{\mathcal{D}}(t,s)$ as $t$ tends to infinity. More precisely, if $m_s^{\mathcal N}(f)$ denotes the average of $f$ with respect to the tight measure $\mu_s^{\mathcal N}$, then, under suitable assumptions we prove that,
for any $R>0$ and any $s\in I$, it holds that
\begin{eqnarray*}
|(G_{\mathcal D}(t,s)f)(x)|\le c_{R,s}e^{\sigma_0(t-s)}\|f\|_{\infty},\qquad\;\, f\in C_b(\Rd_+)
\end{eqnarray*}
and
\begin{eqnarray*}
|(G_{\mathcal N}(t,s)f)(x)-m_s^{\mathcal N}(f)|\le c_{R,s}e^{\sigma_0(t-s)}\|f\|_{\infty},\qquad\;\, f\in C_b(\overline{\Rd_+})
\end{eqnarray*}
for any $(t,x)\in (s,+\infty)\times B_R^+$ and some constants $\sigma_0<0<c_{R,s}$.
The previous pointwise estimates immediately yield
\begin{eqnarray*}
\lim_{t\to +\infty}\|G_{\mathcal D}(t,s)f\|_{L^p(\Rd_+,\mu_t^{\mathcal{N}})}=0,\qquad\;\,\lim_{t\to +\infty}\|G_{\mathcal N}(t,s)f-m_s^{\mathcal{N}}(f)\|_{L^p(\Rd_+,\mu_t^{\mathcal{N}})}=0,
\end{eqnarray*}
for any $f\in L^p(\Rd_+,\mu_s^{\mathcal N})$ and any $p\in (1,+\infty)$.
The construction of the evolution system of measures $\{\mu_t^{\mathcal N}\}_{t\in I}$, as the limit of the tight evolution system of measures associated with $G^{\varepsilon}(t,s)$, is the key tool to deduce many properties of $G_{\mathcal D}(t,s)$ and $G_{\mathcal N}(t,s)$ from the analogous of $G^\varepsilon(t,s)$.
Assuming that the diffusion coefficients do not depend on $x$, we prove both some exponential decay estimates for
$\|G_D(t,s)\|_{L^p(\Rd_+,\mu_t^{\mathcal{N}})}$ and $\|G_{\mathcal N}(t,s)f-m_s^{\mathcal{N}}(f)\|_{L^p(\Rd_+,\mu_t^{\mathcal{N}})}$
and some logarithmic Sobolev inequalities with respect to the measures $\{\mu_t^{\mathcal N}:\;t\in I\}$.
Besides their own interest, the occurrence of logarithmic Sobolev inequalities allows to deduce notable properties such as compactness and hypercontractivity for the evolution operators $G_{\mathcal D}(t,s)$ and $G_{\mathcal N}(t,s)$ as stated in Theorem \ref{thm-other-prop}. Note that, in some sense, the logarithmic Sobolev inequalities are the natural counterpart of the Sobolev embedding theorems that, in general, do not hold when the Lebesgue measure is replaced by evolution systems of measures: consider e.g., the case when $\A(t)$ is the nonautonomous Ornstein-Uhlenbeck operator and the tight evolution system of
measures is of gaussian type.

The paper is organized as follows: in Section \ref{sect-2} we collect some preliminary results. In Section \ref{sect-3} we state and prove the pointwise and uniform gradient estimates for
$G_{\mathcal D}(t,s)$ and $G_{\mathcal N}(t,s)$. In Section \ref{sect-4}, we prove the existence and uniqueness of a tight evolution system of measures for $G_{\mathcal N}(t,s)$, we study the asymptotic behaviour of the evolution operators $G_{\mathcal D}(t,s)$ and $G_{\mathcal N}(t,s)$, we prove the logarithmic Sobolev inequality and some of its consequences. Section \ref{sect-5} contains examples of operators to
which the results of this paper apply. Finally, in the appendix we prove a result which is used in the proof of the pointwise gradient estimates

\subsection*{Notations}
For any $k\ge 0$, we consider the space $C^k_b(\Rd_+)$ (resp. $C^k_b(\overline{\Rd_+})$) consisting all the functions in  $C^k(\Rd_+)$ which are bounded in $\Rd_+$ (resp. in $\overline{\Rd_+}$)
together with all their derivatives (up to the $[k]$-th order) .
We use the subscript ``$c$'' instead of ``$b$''  for spaces of functions  with compact support.
We also consider the space $C_{\mathcal D}^k(\Rd_+)$, $k=0,1$, consisting of functions $f\in C_b^k(\overline{\Rd_+})$ vanishing on $\partial\Rd_+$.

The partial derivatives $\frac{\partial f}{\partial t}$, $\frac{\partial f}{\partial x_i}$ and $\frac{\partial^2f}{\partial x_i\partial x_j}$ are denoted by $D_tf$, $D_if$ and $D_{ij}f$, respectively.
$\textrm{Tr}(Q)$ and $\langle x,y\rangle$ stand for the trace of the square matrix
$Q$ and the Euclidean scalar product
of the vectors $x,y\in\Rd$, respectively.
By $\chi_A$ we denote the characteristic function of the set $A\subset\Rd$ and $\one:=\chi_{\Rd_+}$.
Given a probability measure $\mu$ defined on the Borel $\sigma$-algebra ${\mathcal B}(\Omega)$, we write $\langle\mu,f\rangle$ to
denote the integral of $f\in C_b(\Omega)$ with respect to the measure $\mu$. Somewhere in the paper we find it convenient to split $\Rd\ni x=(x',x_d)$ with $x'\in\R^{d-1}$.
Finally, the Euclidean ball with center at $0$ and radius $r>0$ is denoted by $B_r$ and $B_r^+=B_r\cap \Rd_+$.

\section{Main assumptions and preliminary results}
\label{sect-2}
This section is devoted to prove the existence and uniqueness of a \emph{classical solution} for the Cauchy problems \eqref{NADir} and \eqref{NANeu}.
Here, the term classical has different meanings according to which problem we consider as it is pointed out in the following definition.

\begin{defi}
A function $u\in C^{1,2}((s,+\infty)\times\overline{\R^d_+})$
is called a  bounded classical solution
\begin{enumerate}[\rm (i)]
\item
of the problem \eqref{NADir} if it is bounded and continuous in
$([s,+\infty)\times\overline{\R^d_+})\setminus (\{s\}\times\partial\R^d_+)$ and satisfies \eqref{NADir};
\item
of the problem \eqref{NANeu} if it is bounded and continuous in
$[s,+\infty)\times\overline{\R^d_+}$ and satisfies \eqref{NANeu}.
\end{enumerate}
\end{defi}

Throughout the paper we assume the following outstanding assumptions on the coefficients of the operators $\{\A(t): t\in I\}$, where $I$ is an open right halfline or even $I=\R$.

\begin{hyp}\label{hyp1}
\begin{enumerate}[\rm (i)]
\item
The coefficients $q_{ij}$, $b_j$, $c$ belong to $C^{\alpha/2,1}_{\rm loc}(I\times\overline{\Rd_+})$ for some $\alpha\in (0,1)$ and any $i,j=1,\ldots,d$;
\item
$c_0:= \inf_{I \times \R^d_+}c(t,x) > 0$;
\item
$q_{id}\equiv b_d\equiv 0$ on $I\times\partial\Rd_+$ $(i=1,\ldots,d-1)$;
\item
for every $(t,x)\in I\times \R^d_+$, the matrix $Q(t,x)=[q_{ij}(t,x)]$
is symmetric and there exists a function $\eta:I\times\R^d_+\to\R^+$ such that
$0<\eta_0:=\inf_{I\times\R^d_+}\eta$ and
\begin{align*}
\langle Q(t,x)\xi,\xi\rangle\geq\eta(t,x)|\xi|^2,\qquad\;\, \xi\in \R^d,\;\, (t,x)\in I\times \R^d_+;
\end{align*}
\item
there exist a continuous function $r:I\times \Rd\to \R$ and positive constants $L_0$,  $L_1$ and $R_0$ such that, for any $(t,x)\in I\times \Rd_+$ and any $\xi \in \Rd$,
\begin{align}\label{dissipativity}
\;\;\;\;(i)~r(t,x)\le -L_0 \eta(t,x)+L_1\chi_{B_{R_0}}(x),\qquad\;\, (ii)~\langle \nabla_x b(t,x)\xi,\xi\rangle \leq r(t,x)|\xi|^2;
\end{align}
\item
there exists a positive constant $k_1$ such that
\begin{align}\label{der_qij}
|\nabla_xq_{ij}(t,x)|\leq k_1\eta(t,x),\qquad\;\, (t,x)\in I\times \Rd_+,\;\,
i,j=1,\dots,d.
\end{align}
\end{enumerate}
\end{hyp}

\begin{rmk}
\label{rem-00}
{\rm
Note that Hypotheses \ref{hyp1} imply that, for any bounded set $J\subset I$, there exists a positive constant $\lambda=\lambda_J$ such that
\begin{equation}
\mathcal{A}(t)\varphi(x)\le\lambda_J\varphi(x),\qquad\;\, t\in J,\;\,x\in\Rd_+.
\label{Lyap-0}
\end{equation}
Indeed, note that
$(\mathcal{A}(t)\varphi)(x)=2{\rm Tr}(Q(t,x))+2\langle b(t,x),x\rangle-c(t,x)\varphi(x)$
for any $(t,x)\in I\times\overline{\R^d_+}$. Thanks to \eqref{der_qij} we can estimate
\begin{align*}
{\rm Tr}(Q(t,x))=& {\rm Tr}(Q(t,0))+\sum_{i=1}^d\int_0^1\frac{d}{d\sigma}q_{ii}(t,\sigma x)d\sigma\\
\le & {\rm Tr}(Q(t,0))+|x|\sum_{i=1}^d\int_0^1|\nabla_x q_{ii}(t,\sigma x)|d\sigma\\
\le & {\rm Tr}(Q(t,0))+k_1d|x|\int_0^1\eta(t,\sigma x)d\sigma,
\end{align*}
for any $(t,x)\in I\times\R^d_+$ and any $i=1,\ldots,d$.
Arguing similarly and taking \eqref{dissipativity} into account, we can prove that
\begin{align*}
\langle b(t,x),x\rangle\le & |b(t,0)||x|-L_0|x|^2\int_0^1\eta(t,\sigma x) d\sigma+L_1|x|^2\int_0^1\chi_{B_{R_0}}(\sigma x)d\sigma\\
\le &|b(t,0)||x|-L_0|x|^2\int_0^1\eta(t,\sigma x) d\sigma+L_1\min\{|x|^2,R_0|x|\},
\end{align*}
for any $(t,x)\in I\times\R^d_+$. Summing up, we have
\begin{align*}
(\mathcal{A}(t)\varphi)(x)\leq &2{\rm Tr}(Q(t,0))+ 2(|b(t,0)|+L_1R_0)|x|\\
&+2(dk_1|x|-L_0 |x|^2)\int_0^1\eta(t,\sigma x) d\sigma,
\end{align*}
for any $(t,x)\in I\times\Rd_+$. Observing that $\int_0^1\eta(t,\sigma x)d\sigma\ge \eta_0$ for any $(t,x)\in I\times\Rd_+$, estimate
\eqref{Lyap-0} follows immediately.}
\end{rmk}

\subsection{Approximating evolution operators}
In order to prove the announced existence and uniqueness theorem, we use an approximation procedure. Therefore, considering the
standard reflection with respect to the $x_d$-variable, we define the extension operators
${\mathcal E}, {\mathcal O}: L^{\infty}(\Rd_+)\to L^{\infty}(\Rd)$ by setting
\begin{equation*}
{\mathcal E} f(x):=\left\{
\begin{array}{ll}
f(x',x_d),& x_d\ge 0\\[1mm]
f(x',-x_d),& x_d<0,
\end{array}
\right.
\qquad
{\mathcal O} f(x):=\left\{
\begin{array}{ll}
f(x',x_d),& x_d\ge 0,\\[1mm]
-f(x',-x_d),& x_d<0.
\end{array}
\right.
\end{equation*}
For any function $\psi:I\times \Rd\to \R$ and any $\varepsilon\in (0,1]$, we  denote by
$\psi^\varepsilon:I\times\Rd \to \R$ the convolution (with respect to $x$) of $\psi$
with a standard mollifier $\rho_{\varepsilon}$.

Let $\mathcal{A}^{\varepsilon}(t)$ be the operator defined on smooth functions $\zeta$ by
\begin{equation}
\mathcal{A}^{\varepsilon}(t)\zeta= {\rm Tr}(Q^{\varepsilon}(t,\cdot)D^2\zeta)+\langle b^{\varepsilon}(t,\cdot),\nabla\zeta\rangle-c^{\varepsilon}(t,\cdot)\zeta,\qquad\;\,t\in I,
\label{A-varepsilon}
\end{equation}
where
$q_{ij}^{\varepsilon}=(\tilde q_{ij})^{\varepsilon}$, $b_j^{\varepsilon}=(\tilde b_{j})^{\varepsilon}$ ($i,j=1,\ldots,d$), $c^{\varepsilon}=(\mathcal E c)^{\varepsilon}$ and
\begin{align}
&\tilde{q}_{ij}:=\left\{
\begin{array}{ll}
{\mathcal E}q_{ij},&\;\, i,j<d\; \vee\; i=j=d,\\[1mm]
{\mathcal O}q_{ij},&\;\,i<d, j=d\; \vee\; i=d,\, j<d,
\end{array}
\right.\qquad
\tilde{b}_{i}:=\left\{
\begin{array}{ll}
{\mathcal E}b_{i},&\;\, i<d,\\[1mm]
{\mathcal O}b_{i},&\;\,i=d.
\end{array}
\right.
\label{qij-bj}
\end{align}

\begin{prop}
\label{prop_cup}
For any $\varepsilon\in (0,1]$, any $s\in I$ and any $f\in C_b(\Rd)$ the Cauchy problem
\begin{eqnarray*}
\left\{
\begin{array}{ll}
D_tu(t,x)=\A^{\varepsilon}(t)u(t,x), & t>s,\;\, x\in\Rd,\\[1mm]
u(s,x)=f(x), &x\in\Rd,
\end{array}
\right.
\end{eqnarray*}
admits a unique solution $u^{\varepsilon}\in C_b([s,+\infty)\times\Rd)\cap C^{1+\alpha/2,2+\alpha}_{\rm loc}((s,+\infty)\times\Rd)$. Moreover,
\begin{equation}
\|u^{\varepsilon}(t,\cdot)\|_{\infty}\le e^{-c_0(t-s)}\|f\|_{\infty},\qquad\;\,t>s.
\label{est_inf_eps}
\end{equation}
\end{prop}

\begin{proof}
We begin by observing that $q^{\varepsilon}_{ij}$, $b_{i}^\varepsilon$ and $c^{\varepsilon}$ belong to $C^{\alpha/2,1+\alpha}_{\rm loc}(I\times\Rd)$ for any $i,j=1,\ldots,d$ and they satisfy Hypotheses \ref{hyp1} in $\Rd$, with the same constants $c_0$, $\eta_0$, $k_1$, $L_0$, $L_1$ and with $r$, $\eta$ and $R_0$ being replaced by
$\eta^{\varepsilon}:=({\mathcal E}\eta)^{\varepsilon}$, $r^{\varepsilon}:=({\mathcal E}r)^{\varepsilon}$
and $R_0+1$, respectively. We limit ourselves just to proving that
\begin{equation}
Q^{\varepsilon}\geq \eta^{\varepsilon}I,\qquad\;\,\nabla_x b^{\varepsilon}\le r^{\varepsilon}I,
\label{quadr-forms}
\end{equation}
in the sense of quadratic forms and
that $|\nabla_xq^\varepsilon_{ij}|^2\le k_1^2 ({\eta}^\varepsilon)^2$, since the other properties are straightforward to prove.
For this purpose, we set $\tilde{Q}=(\tilde q_{ij})$ and observe that
$\langle \tilde{Q}(t,x',x_d)\xi,\xi\rangle=\langle Q(t,x',-x_d)(\xi',-\xi_d),(\xi',-\xi_d)\rangle\geq \eta(t,x',-x_d)|\xi|^2$,
for any $t\in I$, $x'\in \R^{d-1}$, any $x_d<0$ and $\xi=(\xi',\xi_d) \in \Rd$. Therefore, we have
\begin{equation}\label{ellip_tilde}
\langle \tilde{Q}(t,x)\xi,\xi\rangle\geq {\mathcal E}\eta(t,x)|\xi|^2,\qquad\;\, \xi\in \R^d,\;\, (t,x)\in I\times \R^d.
\end{equation}
Similarly, since $D_i \tilde{b}_j={\mathcal E}D_ib_j$, if  $i,j<d$ or $i=j=d$, and
$D_i \tilde{b}_j={\mathcal O}D_ib_j$, if $i< d \wedge j=d$ or $i=d \wedge j< d$,
we conclude that
\begin{equation}
\langle \nabla_x \tilde{b}(t,x)\xi,\xi\rangle\le {\mathcal E}r(t,x)|\xi|^2,\qquad\;\, \xi\in \R^d,\;\, (t,x)\in I\times (\R^d\setminus\{0\}).
\label{grad-b}
\end{equation}
Estimates \eqref{ellip_tilde} and \eqref{grad-b} immediately yield the claimed properties on the
matrices $Q^{\varepsilon}$ and $\nabla_x b^{\varepsilon}$.

Finally, since $|D_kq^\varepsilon_{ij}|\le (|{\mathcal E} D_k q_{ij}|)^{\varepsilon}$
for any $i,j,k=1,\ldots,d$, using \eqref{der_qij} and Jensen inequality, we get
\begin{align}
|\nabla_xq^\varepsilon_{ij}(t,x)|^2\le k_1^2 ({\eta}^\varepsilon(t,x))^2,\qquad\;\,(t,x)\in I\times\Rd,\;\,i,j=1,\dots,d.
\label{der_qij-eps}
\end{align}

Thus, the arguments used in Remark \ref{rem-00} show that the function $\varphi$, defined by $\varphi(x)=1+|x|^2$ for any $x\in\Rd$, is a Lyapunov function for the operator $\mathcal{A}^{\varepsilon}(t)$, i.e., for every bounded set $J\subset I$, $\limsup_{|x|\to +\infty}\left(\frac{\mathcal{A}^\varepsilon(t)\varphi}{\varphi}\right)(x)\le -c_J'$,
$c_J'$ being a positive constant, independent of $t\in J$ and of $\varepsilon\in (0,1]$.
Now, \cite[Thm. 2.3]{AngLor10Com} yields the assertion.
\end{proof}

The family of bounded operators $\{G^{\varepsilon}(t,s): t\ge s\in I\}$, defined by $G^{\varepsilon}(t,s)f:=u^{\varepsilon}(t,\cdot)$ for any $t\ge s$, where $u^{\varepsilon}$ is the function in Proposition \ref{prop_cup}, is an evolution operator in $C_b(\Rd)$.
In view of \cite[Thm 2.3 \& Prop. 3.1]{AngLor10Com} there
exists a positive function $g^\varepsilon$ such that
\begin{equation}\label{adesso}
\|g^\varepsilon(t,s,x,\cdot)\|_{L^1(\Rd)}\le e^{-c_0(t-s)},\quad\;\,(t,x)\in (s,+\infty)\times\Rd
\end{equation}
and
\begin{equation}\label{kernerl_eps}
(G^\varepsilon(t,s)f)(x)=\int_{\Rd}f(y)g^\varepsilon(t,s,x,y)dy,\quad\;\,t>s,\,x\in \Rd,
\end{equation}
for any $f\in C_b(\Rd)$.
In particular, from \eqref{adesso} and \eqref{kernerl_eps} we deduce that
\begin{equation}\label{kernerl_eps-bis}
|G^\varepsilon(t,s)(\psi_1\psi_2)|\le (G^\varepsilon(t,s)|\psi_1|^r)^{\frac{1}{r}}
(G^\varepsilon(t,s)|\psi_2|^q)^{\frac{1}{q}},
\end{equation}
for any $\psi_1,\psi_2\in C_b(\Rd)$ and any $r,q\in (1,+\infty)$ such that $1/r+1/q=1$.

\subsection{Existence and uniqueness of the solutions to \eqref{NADir} and \eqref{NANeu}}

In this subsection, we construct by approximation the Dirichlet and the Neumann evolution operators $G_{\mathcal D}(t,s)$ and $G_{\mathcal N}(t,s)$ governing the Cauchy problems \eqref{NADir} and \eqref{NANeu}, respectively.
We begin by stating two maximum principles which immediately yield uniqueness of the classical solutions to \eqref{NADir} and \eqref{NANeu}.

\begin{prop}\label{nsMP}
Fix $s\in I$ and $T>s$. Let $u\in C^{1,2}((s,T]\times\Rd_+)$ be such that
\begin{eqnarray*}
\left\{
\begin{array}{ll}
D_t u(t,x)- {\mathcal L}(t)u(t,x)\leq 0, \quad& (t,x)\in (s,T]\times\Rd_+,\\[1mm]
u(s,x)\le 0, & x\in\Rd_+,
\end{array}
\right.
\end{eqnarray*}
where ${\mathcal L}=\A$ or ${\mathcal L}=\A^{\varepsilon}$ $(\varepsilon\in (0,1])$.
The following properties are satisfied.
\begin{enumerate}[\rm (i)]
\item
If $u\in C_b(([s,T]\times\overline{\Rd_+})\setminus(\{s\}\times\partial\Rd_+))$ and
$u\leq 0$ in $(s,T]\times\partial\Rd_+$, then $u\leq 0$ in  $[s,T]\times\Rd_+$.
\item
If $u\in C_b([s,T]\times\overline{\Rd_+})\cap C^{0,1}((s,T]\times\overline{\Rd_+})$ and
$\displaystyle\frac{\partial u}{\partial\nu}\leq 0$ in $(s,T]\times\partial\Rd_+$, then $u\leq 0$ in  $[s,T]\times\Rd_+$.
\end{enumerate}
\end{prop}
\begin{proof}
The assertions can be obtained adapting to the nonautonomous setting the proofs in \cite[Thm. A.2]{ForMetPri04Gra} and in \cite[Prop. 2.1]{BerFor04Grad}, using $\varphi(x)=1+|x|^2$ as a Lyapunov function.
\end{proof}

\begin{thm}\label{existence}
For any $s\in I$ and $f\in C_b(\R^d_+)$ $($resp. $f\in C_b(\overline{\R^d_+}))$ the problem \eqref{NADir} $($resp. \eqref{NANeu}$)$ admits a unique bounded classical solution $u_{\mathcal D}$ $($resp. $u_{\mathcal N})$.
Moreover, $u_{\mathcal D}$ and $u_{\mathcal N}$ belong to $C^{1+\alpha/2,2+\alpha}_{\rm loc}((s,+\infty)\times\overline{\R^d_+})$, they satisfy the estimates
\begin{equation}
\label{estinf}
(i)~\|u_{\mathcal D}(t,\cdot)\|_\infty\leq e^{-c_0(t-s)}\| f\|_\infty,\qquad\;\,(ii)~\|u_{\mathcal N}(t,\cdot)\|_\infty\leq e^{-c_0(t-s)}\| f\|_\infty,
\end{equation}
for any $t>s$, and they are nonnegative if $f\ge 0$. Moreover, if $f\in C^{2+\alpha}_c(\Rd)$, then $u_{\mathcal D}$ and $u_{\mathcal N}$ belong to $C^{1+\alpha/2,2+\alpha}_{\rm loc}([s,+\infty)\times\overline{\Rd_+})$.
\end{thm}

\begin{proof}
The uniqueness part and the non-negativity of $u_{\mathcal D}$ and $u_{\mathcal N}$, when $f\ge 0$, follow from Proposition \ref{nsMP}. The existence of a solution will be proved in some steps.
We begin by considering the Cauchy Dirichlet problem \eqref{NADir}.

{\em Step 1.} Here, we prove that, for any $f\in C_{\mathcal D}(\Rd_+)$, the unique classical solution to the Cauchy-Dirichlet problem
\eqref{NADir} with $\A(t)$ being replaced by $\A^{\varepsilon}(t)$, which we denote by $u^{\varepsilon}_{\mathcal D}$, is the restriction to $\Rd_+$ of
the function $G^{\varepsilon}(\cdot,s){\mathcal O}f$, i.e.,
\begin{equation}
G^{\varepsilon}_{\mathcal D}(t,s)f=
(G^{\varepsilon}(t,s){\mathcal O}f)_{|\Rd_+},\qquad\;\,t>s.
\label{Dir_eps}
\end{equation}
Clearly, the function in the right-hand side of \eqref{Dir_eps} solves the differential equation and satisfies the initial condition in \eqref{NADir}. To prove that it vanishes on $(s,+\infty)\times\partial\Rd_+$, we show that, if $\psi\in C_b(\Rd)$ is odd with respect to the variable $x_d$, then, for any $s\in I$,
$G^{\varepsilon}(t,s)\psi$ is odd with respect to the variable $x_d$. This clearly implies that $G^{\varepsilon}(t,s)\psi$ vanishes on $\partial\Rd_+$.
To check this property, observe that the function $v\in C_b([s,+\infty)\times\Rd)\cap C^{1,2}((s,+\infty)\times\Rd)$, defined by
$v(t,x)=(G^{\varepsilon}(t,s)\psi)(x_1,\ldots,x_{d-1},-x_d)$ for any $t>s$ and any $x\in\Rd$, solves the equation $v_t-{\mathcal A}^{\varepsilon}(t)v=0$
in $(s,+\infty)\times\Rd$, due to the symmetry properties of the coefficients of the operator $\A^{\varepsilon}(t)$. Since $v(s,\cdot)=-\psi$ in $\Rd$, the uniqueness of the solution to the Cauchy
problem
\begin{eqnarray*}
\left\{
\begin{array}{ll}
D_tw(t,x)=\A^{\varepsilon}(t)w(t,x), & t>s,\;\, x\in\Rd,\\[1mm]
w(s,x)=-\psi(x), &x\in\Rd,
\end{array}
\right.
\end{eqnarray*}
which follows from Proposition \ref{prop_cup},
guarantees that $v=-G^{\varepsilon}(\cdot,s)\psi$, and this yields the claim.

{\em Step 2.} Here, we prove the existence of a classical solution to \eqref{NADir} in the case when $f\in C^{2+\alpha}_c(\Rd_+)$.
From the classical Schauder estimates we can infer that, for any $k\in\N$, there exists a positive constant $c_k$, depending only on ${\eta}^\varepsilon$, the $C^{\alpha/2,\alpha}$-norms of the coefficients of the operator $\mathcal A^{\varepsilon}(t)$ in $[s,s+k]\times B_{2k}^+$, such that
\begin{equation}
\|u^{\varepsilon}_{\mathcal D}\|_{C^{1+\alpha/2,2+\alpha}((s,s+k)\times B_k^+)}\le c_k\|f\|_{C^{2+\alpha}_c(\Rd_+)}.
\label{nonsichiama}
\end{equation}
Note that the constant $c_k$ can be taken independent of $\varepsilon$ since $\eta^{\varepsilon}\ge\eta_0$ and the $C^{\alpha/2,\alpha}$-norms of the coefficients $q_{ij}^{\varepsilon}$, $b_j^{\varepsilon}$ and $c^{\varepsilon}$ ($i,j=1,\ldots,d$) in $(s,s+k)\times B_{2k}^+$ can be estimated from above, uniformly with respect to $\varepsilon\in (0,1]$, in terms of the
$C^{\alpha/2,\alpha}$-norms of $q_{ij}$, $b_j$ and $c$ ($i,j=1,\ldots,d$) in the same set.

In view of Arzel\`a-Ascoli theorem and \eqref{nonsichiama}, for any  $k\in\N$ there exist an infinitesimal sequence $(\varepsilon_n^k)\subset (0,1)$ and a function
$u_k\in C^{1+\alpha/2,2+\alpha}((s,s+k)\times B_k^+)$ such that $u^{\varepsilon_n^k}_{\mathcal D}$ converges to $u_k$ in
$C^{1,2}([s,s+k]\times\overline{B_k^+})$ as $n\to +\infty$.
Without loss of generality, we can assume that $(\varepsilon_n^{k+1})\subset (\varepsilon_n^{k})$ for any $k\in\N$.
Hence, by a diagonal argument, we can find an infinitesimal sequence $(\varepsilon_n)$ such that
$u^{\varepsilon_n}_{\mathcal D}$ converges to $u$ in
$C^{1,2}([s,s+k]\times\overline{B_k^+})$ as $n\to +\infty$, for any $k\in\N$, where $u:[s,+\infty)\times\overline{\Rd_+}\to\R$ is defined by
$u(t,x)=u_k(t,x)$, $k$ being any integer such that $(t,x)\in [s,s+k]\times \overline{B_k^+}$.
Clearly, $u\in C^{1+\alpha/2,2+\alpha}_{\rm loc}([s,+\infty)\times\overline{\Rd_+})$ is a bounded classical solution to problem \eqref{NADir} and it satisfies \eqref{estinf} thanks to \eqref{est_inf_eps}.

{\em Step 3.} We now fix $f\in C_{\mathcal D}(\Rd_+)$ which tends to zero at infinity, and consider a sequence $(f_n)\subset C^{2+\alpha}_c(\Rd_+)$ converging to $f$ uniformly in $\Rd_+$.
By Step 1, for any $n\in\N$, the Cauchy problem \eqref{NADir}, with $f$ being replaced by $f_n$ admits a unique solution
$u_n\in C^{1+\alpha/2,2+\alpha}_{\rm loc}([s,+\infty)\times\overline{\Rd_+})$.
Interior Schauder estimates show that, for any $R,T>0$ and $\sigma<T-s$, there exists a positive constant $C$, independent of $n$, such that
$\|u_n\|_{C^{1+\alpha/2,2+\alpha}((s+\sigma,T)\times B_R^+)}\le C\|f\|_{\infty}$ for  any $n\in\N$.
Arguing as in Step 1, we prove that $u_n$ converges to a function $u$ which belongs to $C^{1+\alpha/2,2+\alpha}_{\rm loc}((s,+\infty)\times\overline{\Rd_+})$, vanishes on $[s,+\infty)\times\partial\Rd_+$ and
satisfies $D_tu=\A(t)u$ in $(s,+\infty)\times\Rd_+$. Moreover, since
$\|u_n(t,\cdot)-u_m(t,\cdot)\|_{\infty}\le e^{-c_0(t-s)}\|f_n-f_m\|_{\infty}$ for any $t>s$ and 
any $m,n\in\N$, $u$ actually belongs to $C_b([s,+\infty)\times\overline{\Rd_+})$ and $u(s,\cdot)=f$ since $u_n(s,\cdot)=f$ for any $n\in\N$.
Hence, $u$ is a classical solution to problem \eqref{NADir} and, of course, it satisfies estimate \eqref{estinf}.

{\em Step 4.} Finally, we deal with the general case when $f\in C_b(\Rd_+)$  and denote by $u_g$ the unique solution to problem \eqref{NADir} with initial datum $g\in C_{\mathcal D}(\Rd_+)$ which tends to zero at infinity. We consider a sequence of functions $(f_n)\in C^{2+\alpha}_c(\Rd_+)$ converging to $f$ locally uniformly in $\Rd_+$ and such that
$M:=\sup_{n\in\N}\|f_n\|_{\infty}<+\infty$. The already used compactness argument shows that, up to a subsequence, $u_{f_n}$ converges in $C^{1,2}_{\rm loc}((s,+\infty)\times\overline{\Rd_+})$ to
a function $u\in C^{1+\alpha/2,2+\alpha}_{\rm loc}((s,+\infty)\times\overline{\Rd_+})$. Hence, $u$ satisfies the differential equation, the boundary condition in \eqref{NADir},
and also the estimate \eqref{estinf}.

To complete the proof, we show that $u$ is continuous also on $\{s\}\times\Rd_+$  and $u(s,\cdot)=f$ in $\Rd_+$.
For this purpose, we fix a compact set $K\subset \Rd_+$ and a smooth and compactly supported function $\psi$
such that $0\leq \psi \leq 1$ and $\psi\equiv 1$ in $K$. Since $f_n=\psi f_n+(1-\psi)f_n$ for every $n \in \N$, by linearity
$u_{f_n}=u_{\psi f_n}+ u_{(1-\psi)f_n}$.
We know that the functions $u_{\psi f_n}$ and $u_{\psi}$
are continuous up to $s$ where they are equal to $\psi f_n$ and
$\psi$ respectively. Proposition \ref{nsMP}(i) and the positivity of $c$ yield
$\|u_{(1-\psi)f_n}\|_{\infty}\leq M( 1-u_{\psi})$ for any $n\in\N$.
Hence,
$|u_{f_n}- f| \leq |u_{\psi f_n}-\psi f|+M(1-u_{\psi})$
in $(s,+\infty)\times K$. Letting $n\to +\infty$ we obtain
$|u-f| \leq |u_{\psi f}-\psi f| +  M(1-u_{\psi})$
in the same set as above.
Now, it follows that $u$ can be extended by continuity at $t=s$ by setting $u(s,x)=f(x)$ for any $x\in K$.
By the arbitrariness of $K$ we deduce that $u$ is continuous on $\{s\}\times\Rd_+$ and $u(s,\cdot)=f$.
\medskip

The proof of the claim in the case of the Cauchy-Neumann problem \eqref{NANeu} follows the same lines of the Dirichlet case, taking into account that,
for any $f\in C_b(\overline{\Rd_+})$, the unique classical solution to the Cauchy-Neumann problem
\eqref{NANeu} with $\A(t)$ being replaced by $\A^{\varepsilon}(t)$ is the restriction to $\Rd_+$ of
the function $G^{\varepsilon}(\cdot,s){\mathcal E}f$, i.e.,
\begin{equation}
G^{\varepsilon}_{\mathcal N}(t,s)f=
(G^{\varepsilon}(t,s){\mathcal E}f)_{|\Rd_+},\qquad\;\,t>s.
\label{Neu_eps}
\end{equation}
Moreover, the continuity of the solution of problem \eqref{NANeu} on $\{s\}\times\overline{\Rd_+}$ follows from observing that, in the analogous of Step 4, we can consider a sequence of functions $f_n\in C^{2+\alpha}_c(\overline{\Rd_+})$ converging to $f$ locally uniformly in $\overline{\Rd_+}$, and the compact set $K$ can be a subset of $\overline{\Rd_+}$.
\end{proof}


In view of Theorem \ref{existence} we can define two families of
bounded linear operators $\{G_{\mathcal D}(t,s):\,t\ge s\in I\}$ in $C_b(\R^d_+)$ and
$\{G_{\mathcal N}(t,s):\,t\ge s\in I\}$ in $C_b(\overline{\R^d_+})$, by setting
\begin{align*}
&(G_{\mathcal D}(t,s)f)(x):= u_{\mathcal D}(t,x),\qquad\;\,(t,x)\in [s,+\infty)\times\R^d_+,\\
&(G_{\mathcal N}(t,s)f)(x):= u_{\mathcal N}(t,x),\qquad\;\,(t,x)\in [s,+\infty)\times\overline{\R^d_+}.
\end{align*}
The evolution laws
$G_{\mathcal D}(t,s)=G_{\mathcal D}(t,r)G_{\mathcal D}(r,s)$ and $G_{\mathcal N}(t,s)=G_{\mathcal N}(t,r)G_{\mathcal N}(r,s)$, for any $I\ni s\leq r\leq t$,
are immediate consequence of the uniqueness of the solutions to problems \eqref{NADir} and \eqref{NANeu}.
The families $\{G_{\mathcal D}(t,s):\,t\ge s\in I\}$ and $\{G_{\mathcal N}(t,s):\,t\ge s\in I\}$ are called the
{\emph{evolution operator associated to problem \eqref{NADir} and \eqref{NANeu}}}, respectively. In the sequel, to lighten the notation,
we simply write $G_{\mathcal D}(t,s)$ and $G_{\mathcal N}(t,s)$ to denote the previous two evolution operators.

In the following proposition we collect some useful properties of $G_{\mathcal D}(t,s)$ and $G_{\mathcal N}(t,s)$.

\begin{prop}\label{nonesiste}
Fix $s\in I$. The following statements are satisfied:
\begin{enumerate}[\rm (i)]
\item
if $(f_n)\subset C_b(\Rd_+)$ $($resp. $(f_n)\subset C_b(\overline{\Rd_+}))$
 is a bounded sequence, with respect to the sup-norm, which converges to $f\in C_b(\Rd_+)$ $($resp. $f\in C_b(\overline{\Rd_+}))$ locally uniformly in $\Rd_+$,
then $G_{\mathcal D}(\cdot,s)f_n$ $($resp.
$G_{\mathcal N}(\cdot,s)f_n)$
converges to $G_{\mathcal D}(\cdot,s)f$ $($resp. $G_{\mathcal N}(\cdot,s)f)$ in $C^{1,2}_{\rm loc}((s,+\infty)\times\overline{\Rd_+})$;
\item
if $f\in C_b(\overline{\Rd_+})$ is nonnegative, then $G_{\mathcal D}(t,s)f\leq G_{\mathcal N}(t,s)f$ for every $t>s$.
\end{enumerate}
\end{prop}
\begin{proof}
Property (i) with ${\mathcal I}={\mathcal D}$ is a byproduct of Step 4 in the proof of Theorem \ref{existence}, and, when $\mathcal J=\mathcal N$,
its proof is completely similar.

Let us prove property (ii). Fix a nonnegative function $f\in C_b(\overline{\Rd_+})$.
By Theorem \ref{existence}, $G_{\mathcal N}(\cdot,s)f$ is nonnegative in $(s,+\infty)\times\overline{\Rd_+}$.
Now, we consider the function $v=G_{\mathcal D}(\cdot,s)f-G_{\mathcal N}(\cdot,s)f$. Clearly,
$v\in C^{1,2}((s,+\infty)\times\overline{\R^d_+})\cap C_b(([s,+\infty)\times\overline{\R^d_+})\setminus (\{s\}\times\partial\Rd_+))$, it solves $D_t v-\mathcal{A}(t)v=0$ in $(s,+\infty)\times \Rd_+$ and $v\equiv 0$ in $\{s\}\times \Rd_+$. Finally, since $v=-G_{\mathcal N}(\cdot,s)f\le 0$ in $(s,+\infty)\times\partial \Rd_+$,
we conclude, using Proposition \ref{nsMP}, that $v\le 0$ in $(s,+\infty)\times \Rd_+$, i.e., $G_{\mathcal D}(\cdot,s)f\le G_{\mathcal N}(\cdot,s)f$.
\end{proof}

\section{Gradient estimates}
\label{sect-3}
In this section we provide both pointwise and uniform (spatial) gradient estimates for the functions $G_{\mathcal D}(t,s)f$ and $G_{\mathcal N}(t,s)f$, when $f\in C_b(\Rd_+)$ and $f\in C_b(\overline{\Rd_+})$, respectively, and when $f$ is even much smoother.
If not otherwise specified, throughout this section we assume that the following conditions are satisfied.

\begin{hyp}\label{hyp2}
\begin{enumerate}[\rm (i)]
\item
Hypotheses $\ref{hyp1}$ are satisfied;
\item
there exist a continuous function $\beta:I\times \Rd_+\to [0,+\infty)$ and
a positive constant $k_2$ such that
\begin{align}\label{cond-deriv-c}
\;\;\;\;\;\;\;\;(i)~\beta\leq k_2 c,\qquad\;\, (ii)~|\nabla_xc|\le \beta,\quad\;\,{\rm in}\; I\times\Rd_+;
\end{align}
\item
for every $p\in(1,+\infty)$ there exists a positive constant $C_p$ such that
\begin{align}
\;\;\;\;\;\;\;r+\left(\frac{k_1^2d^2}{4M_p}-M_p\right)\eta-\left (1-\frac{1}{p}\right )c+\frac{pk_2}{4(p-1)}\beta\leq C_p,\,
\label{cond-r-eta-c}\end{align}
in $I\times \Rd_+$, where $M_p=\min\{1,p-1\}$.
\end{enumerate}
\end{hyp}

\begin{rmk}
{\rm Note that the coefficients $q^{\varepsilon}_{ij}$, $b_{i}^\varepsilon$ ($i,j=1,\ldots,d$) and $c^{\varepsilon}$ satisfy Hypotheses \ref{hyp2}(ii)-(iii) with
$\Rd_+$ being replaced by $\Rd$, with the same constants $k_1$, $k_2$, $M_p$, $C_p$, and with $\eta$, $r$ and $\beta$ being replaced by $\eta^{\varepsilon}$, $r^{\varepsilon}$ and
 $\beta^{\varepsilon}:=({\mathcal E}\beta)^{\varepsilon}$, respectively.}
 \end{rmk}

\subsection{$C^1$-$C^1$ uniform and pointwise estimates}
As it has been already remarked in the introduction, we can not expect an estimate where $|\nabla_xG_{\mathcal D}(t,s)g|^p$ is controlled from above by $G_{\mathcal D}(t,s)(|g|^p+|\nabla g|^p)$ since this latter vanishes on $\partial \Rd_+$. However, for any function $g\in C^1_{\mathcal D}(\Rd_+)$, we can estimate $|\nabla_x G_{\mathcal D}(t,s)g|^p$ by means of $G_{\mathcal N}(t,s)(|g|^p+|\nabla g|^p)$,
 as the next theorem shows.

\begin{thm}\label{Dir_thm}
The pointwise gradient estimate
\begin{eqnarray}\label{GD_GN}
|\nabla_x G_{\mathcal I}(t,s)f|^p\leq  2^{\ell_p+p-1}e^{C_p(t-s)}G_{\mathcal N}(t,s)(|f|^p+|\nabla f|^p),
\end{eqnarray}
holds in $\overline{\Rd_+}$ for any $t>s\in I$, any $p\in (1,+\infty)$ and any $f\in C^1_b(\overline{\Rd_+})$, when ${\mathcal I}={\mathcal N}$, and any $f\in C_{\mathcal D}^1(\Rd_+)$
when ${\mathcal I}={\mathcal D}$.  Here,
$C_p$ is the constant in \eqref{cond-r-eta-c} and $\ell_p=\max\{p/2-1,1\}$.
\end{thm}
\begin{proof}
The core of the proof  consists in proving the gradient estimate
\begin{align}\label{grad_est_rd}
|\nabla_x G^\varepsilon (t,s)f|^p\leq 2^{\ell_p+p-1}e^{C_p(t-s)}G^\varepsilon(t,s)(|f|^p+|\nabla f|^p),
\end{align}
in $\Rd$ for any $t>s$, any function $f\in C^{3+\alpha}_c(\Rd)$ and any $\varepsilon\in (0,1]$.
This estimate is obtained in Steps 1 and 2. More precisely, in Step 1 we prove that, for any $t>s\in I$ and any $n\in\N$,
\begin{align}\label{grad_est_rd-0}
|\nabla_x G^\varepsilon_{{\mathcal N},n} (t,s)f|^p\leq 2^{\ell_p}e^{C_p(t-s)}G^\varepsilon_{{\mathcal N},n}(t,s)(|f|^p+|\nabla f|^p),
\end{align}
in $B_n$ for positive functions $f\in C^{3+\alpha}(\Rd)$ which are constant outside a compact set contained in $B_n$, where $G^{\varepsilon}_{{\mathcal N},n}(t,s)$ denotes the evolution operator
associated to the restriction of the operator $\mathcal{A}^{\varepsilon}(t)$ (see \eqref{A-varepsilon}) to $B_n$, with homogeneous Neumann boundary conditions. In the second one, we complete the proof of \eqref{grad_est_rd}.
Finally, in Step 3, we prove \eqref{GD_GN}.

{\em Step 1.}
Fix $s \in I$ and $\varepsilon\in (0,1]$; for any positive function $f\in C^{3+\alpha}(\Rd)$, which is constant outside the ball $B_{n_0}$, we set
$u^{\varepsilon}_n: =G^{\varepsilon}_{{\mathcal N},n}(\cdot,s)f$ for any $n> n_0$
 and consider the function
$w=[(u^{\varepsilon}_n)^2+|\nabla_x u^{\varepsilon}_n|^2]^{p/2}$ which belongs to $C_b([s,+\infty) \times\overline{B_n})\cap C^{1,2}((s,+\infty) \times B_n)$ (see Proposition \ref{smoothdatum}).
Since $u^{\varepsilon}_n(t,x)\ge\delta$ for any $t>s$, $x\in B_n$ and some $\delta>0$, $w$ has positive infimum in $(s,+\infty)\times B_n$. Moreover,
$w_t - \A(t)w = \psi_{1,p} + \psi_{2,p} + \psi_{3,p}+\psi_{4,p}$,
where
\begin{align}
\psi_{1,p}=& p w^{1-\frac{2}{p}}\bigg (\sum_{i,j,k=1}^d
D_kq_{ij}^{\varepsilon}D_ku^{\varepsilon}_nD_{ij}u^{\varepsilon}_n-u^{\varepsilon}_n
\langle\nabla_xc^{\varepsilon} ,\nabla_xu^{\varepsilon}_n\rangle\bigg ),
\label{psi-1}\\[3mm]
\psi_{2,p} =& -p(p-2)w^{1-\frac{4}{p}}
|\sqrt{Q^{\varepsilon}}(u^{\varepsilon}_n\nabla_xu^{\varepsilon}_n+D^2_xu^{\varepsilon}_n\nabla_xu^{\varepsilon}_n)|^2,
\label{psi-2}
\\[3mm]
\psi_{3,p} =& pw^{1-\frac{2}{p}}\bigg (\langle \nabla_x b^{\varepsilon}\, \nabla_x u^{\varepsilon}_n , \nabla_x u^{\varepsilon}_n \rangle
\hskip -.5truemm -\hskip -.5truemm |\sqrt{Q^{\varepsilon}}\nabla_xu_n^{\varepsilon}|^2
- \sum_{k=1}^d
|\sqrt{Q^{\varepsilon}}\nabla_x D_k u^{\varepsilon}_n|^2\bigg ),
\label{psi-3}\\
\psi_{4,p}=& -(p-1)c^{\varepsilon}w.
\label{psi-4}
\end{align}
We are going to prove that
\begin{align}
w_t - \A^{\varepsilon}(t)w \leq pC_p w.
\label{montecarlo-0}
\end{align}
For this purpose, we begin by observing that, from \eqref{der_qij-eps} and \eqref{cond-deriv-c} we obtain
\begin{align}
\psi_{1,p}
&\le pw^{1-\frac{2}{p}}\left(dk_1\eta^{\varepsilon}|\nabla_x u^{\varepsilon}_n| |D_x^2u^{\varepsilon}_n| + \beta^\varepsilon |u^{\varepsilon}_n||\nabla_x u^{\varepsilon}_n| \right)\notag\\
&\le pw^{1-\frac{2}{p}}\left [a_1dk_1\eta^{\varepsilon}|D_x^2u^{\varepsilon}_n|^2+
\left (\frac{dk_1}{4a_1}\eta^{\varepsilon}+\frac{\beta^{\varepsilon}}{4a_2}\right )
|\nabla_x u^{\varepsilon}_n|^2 + a_2\beta^{\varepsilon}(u^{\varepsilon}_n)^2\right ],
\label{fermagli}
\end{align}
for any $a_1,a_2>0$.

If $p\ge 2$, $\psi_{2,p}$ is nonpositive in $I\times B_n$.
Moreover, taking \eqref{quadr-forms} into account, we deduce that $\psi_{3,p}\le pw^{1-\frac{2}{p}}\big [( r^{\varepsilon}-\eta_{\varepsilon})\, |\nabla_x u^{\varepsilon}_n|^2 - \eta^{\varepsilon}\, |D_x^2u^{\varepsilon}_n|^2\big ]$. Hence, the above computations yield
\begin{align*}
&\psi_{1,p}+\psi_{2,p}+\psi_{3,p}+\psi_{4,p}\\
\le &w^{1-\frac{2}{p}}\bigg [p\left (a_1dk_1-1\right )\eta^{\varepsilon}|D_x^2u^{\varepsilon}_n|^2+(pa_2\beta^{\varepsilon}-(p-1)c^{\varepsilon})(u^{\varepsilon}_n)^2\\
&\qquad\;\;\,+\bigg (p r^{\varepsilon}+p\bigg (\frac{dk_1}{4a_1}-1\bigg )\eta^{\varepsilon}+p\frac{\beta^{\varepsilon}}{4a_2}-(p-1)c^{\varepsilon}\bigg )
|\nabla_x u^{\varepsilon}_n|^2\bigg ],
\end{align*}
for $p\ge 2$ and any $a_1, a_2>0$.
Choosing $a_1=M_p(dk_1)^{-1}$, $a_2=(p-1)(pk_2)^{-1}$, recalling that $\beta^{\varepsilon}\le k_2 c^{\varepsilon}$ and observing that condition \eqref{cond-r-eta-c} implies that
\begin{eqnarray*}
r^{\varepsilon}+\left(\frac{k_1^2d^2}{4M_p}-M_p\right)\eta^{\varepsilon}-\left (1-\frac{1}{p}\right )c^{\varepsilon}+\frac{pk_2}{4(p-1)}\beta^{\varepsilon}\leq C_p,
\end{eqnarray*}
we deduce that  $w_t - \A^{\varepsilon}(t)w
\le pC_p w^{1-\frac{2}{p}}|\nabla_x u^{\varepsilon}_n|^2\leq pC_p w$, i.e.,
\eqref{montecarlo-0} follows.

If $1<p<2$ we use the triangle inequality $|\sqrt{Q^{\varepsilon}}(\lambda+\mu)|\le |\sqrt{Q^{\varepsilon}}\lambda|+|\sqrt{Q^{\varepsilon}}\mu|$,
with $\lambda=u^{\varepsilon}_n\nabla_xu^{\varepsilon}_n$ and $\mu=D_x^2u^{\varepsilon}_n\nabla_xu^{\varepsilon}_n$ to estimate
\begin{align*}
\psi_{2,p}\le p(2\hskip -.5mm - \hskip -.5mm p)w^{1-\frac{4}{p}}\hskip -.5mm\Bigg [u^{\varepsilon}_n|\hskip -.5truemm\sqrt{Q^{\varepsilon}}\nabla_xu^{\varepsilon}_n|\hskip -1mm
+\hskip -1mm\Bigg (\hskip -.5mm\sum_{h,k=1}^d\hskip -1mm D_hu^{\varepsilon}_nD_ku^{\varepsilon}_n\langle Q^{\varepsilon}\nabla_xD_hu^{\varepsilon}_n,\nabla_xD_ku^{\varepsilon}_n\rangle\hskip -1mm\Bigg )^{\hskip -1truemm\frac{1}{2}}\Bigg ]^2.
\end{align*}
Using twice the Cauchy-Schwarz inequality we get
\begin{align*}
\psi_{2,p}\le &p(2-p)w^{1-\frac{4}{p}}\Bigg [u^{\varepsilon}_n |\sqrt{Q^{\varepsilon}}\nabla_xu^{\varepsilon}_n|
+|\nabla_xu^{\varepsilon}_n|\Bigg (\sum_{h=1}^d|\sqrt{Q^{\varepsilon}}\nabla_xD_hu^{\varepsilon}_n|^2\Bigg )^{\frac{1}{2}}
\Bigg ]^2\nonumber\\
\le & p(2-p)w^{1-\frac{2}{p}}\left (
|\sqrt{Q^{\varepsilon}}\nabla_xu^{\varepsilon}_n|^2+\sum_{h=1}^d|\sqrt{Q^{\varepsilon}}\nabla_xD_hu^{\varepsilon}_n|^2\right ).
\end{align*}
It thus follows that
\begin{align}
\psi_{2,p}+\psi_{3,p}\le &pw^{1-\frac{2}{p}}\Bigg [\langle \nabla_xb^{\varepsilon}\nabla_xu^{\varepsilon}_n,\nabla_xu^{\varepsilon}_n\rangle\notag\\
&\qquad\;\;\;\;\;\;+(1-p)\left (|\sqrt{Q^{\varepsilon}}\nabla_xu^{\varepsilon}_n|^2+\sum_{h=1}^d|\sqrt{Q^{\varepsilon}}\nabla_xD_hu^{\varepsilon}_n|^2\right )\Bigg ]\notag\\
\le & pw^{1-\frac{2}{p}}\left [ (r^{\varepsilon}+(1-p)\eta^{\varepsilon})|\nabla_xu^{\varepsilon}_n|^2+(1-p)\eta^{\varepsilon}|D^2_xu^{\varepsilon}_n|^2\right ].
\label{colors}
\end{align}
 From \eqref{fermagli}-\eqref{colors} we get inequality  \eqref{montecarlo-0} also in the case $p\in (1,2)$.

To complete the proof of \eqref{grad_est_rd-0}, we denote by $\nu$ the unit exterior normal vector to $\partial B_n$ and observe that
\begin{eqnarray*}
\frac{\partial w}{\partial\nu}(t,x)=p(w(t,x))^{1-\frac{2}{p}}\langle D^2_xu_n^{\varepsilon}(t,x)\nabla_xu_n^{\varepsilon}(t,x),\nu(x)\rangle,\quad\;\,(t,x)\in (s,+\infty)\times\partial B_n,
\end{eqnarray*}
which is nonpositive since the domain is convex (see e.g., \cite[Sect. V.B]{GaHuLa90RU}). Hence, the function  $v = w-e^{C_p(\cdot-s)}G^{\varepsilon}_{{\mathcal N},n}(\cdot,s) (f^2+|\nabla f|^2)^{\frac{p}{2}}$ satisfies
$v_t - (\A^{\varepsilon}(t)+C_p)v\le 0$ in $(s,+\infty)\times B_n$, it vanishes on $\{s\}\times\overline{B_n}$ and
its normal derivative is nonpositive in $(s,+\infty)\times\partial B_n$.
The classical maximum principle
implies that  $v\le 0$, and estimate \eqref{grad_est_rd-0} follows at once.

{\em Step 2.}
In view of \cite[Thm. 2.3(i)]{AngLor10Com}, $G^{\varepsilon}_{{\mathcal N},n}(t,s)f$ converges to $G^{\varepsilon}(t,s)f$ in $C^{1,2}(D)$ for any compact set $D\subset (s,+\infty)\times\Rd$. Hence, letting
$n\to +\infty$ in \eqref{grad_est_rd-0} we get
\begin{align}\label{grad_est_rd-1}
|\nabla_x G^\varepsilon(t,s)f|^p\leq 2^{\ell_p}e^{C_p(t-s)}(G^\varepsilon(t,s)(|f|^p+|\nabla f|^p),
\end{align}
for any $t>s$ and any positive function $f$ as in Step 1.

Estimate \eqref{grad_est_rd-1} can be extended easily to any nonnegative $f\in C^{3+\alpha}_c(\Rd)$ by a density argument, approximating $f$
by the sequence of function $(f_n)$ defined by $f_n=f+\frac{1}{n}$ for any $n\in\N$.


For a general function $f\in C^{3+\alpha}_c(\Rd)$, we split $f=f^+-f^-$, where $f^{\pm}=\max\{\pm f,0\}$. Clearly, $f^+$ and $f^-$ are bounded and Lipschitz continuous in $\Rd$.
Moreover, $\nabla f^+=\chi_{\{f>0\}}\nabla f$ and $\nabla f^-=\chi_{\{f<0\}}\nabla f$. In particular, $|\nabla f^{\pm}|\le |\nabla f|$.
Let us consider the sequences $(g_n^+)$ and $(g_n^-)$, where $g_n^{\pm}:=f^{\pm}\star\rho_{1/n}$ for any $n\in\N$, and $(\rho_{1/n})$ is a standard sequence of mollifiers.
Notice that $g_n^{\pm}\in C^1_b(\Rd)$ are nonnegative, for any $n\in\N$, and converge uniformly in $\Rd$ to $f^{\pm}$ as $n\to +\infty$. Moreover, up to a subsequence, we can assume that
$\nabla g_n^{\pm}$ converge pointwise  a.e. in $\Rd$ to $\nabla f^{\pm}$ as $n\to +\infty$. Hence, we can write
\begin{equation}\label{con_n-g}
|\nabla_x G^\varepsilon (t,s)g_n^{\pm}|^p\leq 2^{\ell_p}e^{C_p(t-s)}G^\varepsilon(t,s)[(g_n^{\pm})^p+|\nabla g_n^{\pm}|^p],
\end{equation}
for any $t>s$ and any $n\in\N$. Arguing as above we can show that $|\nabla_x G^\varepsilon (t,s)g_n^{\pm}|$ converges to
$|\nabla_x G^\varepsilon (t,s)f^{\pm}|$, locally uniformly in $\Rd$.
Similarly, using \eqref{adesso} and dominated convergence we conclude that also the right-hand side of  \eqref{con_n-g} tends to
\begin{align*}
2^{p-1}e^{C_p(t-s)}\int_{\Rd}((f^{\pm}(y))^p+|\nabla f^{\pm}(y)|^p)g^{\varepsilon}(t,s,x,y)dy,
\end{align*}
pointwise in $\Rd$.
Therefore,
\begin{align*}
|\nabla_x G^\varepsilon (t,s)f^{\pm}|^p\leq &2^{\ell_p}e^{C_p(t-s)}\int_{\Rd}((f^{\pm}(y))^p+|\nabla f^{\pm}(y)|^p)g^{\varepsilon}(t,s,x,y)dy\\
\le & 2^{\ell_p}e^{C_p(t-s)}G^\varepsilon(t,s)(|f|^p+|\nabla f|^p),
\end{align*}
for any $t>s$.
This estimate yields \eqref{grad_est_rd}.

{\em Step 3.}
Fix $t>s\in I$. First we assume ${\mathcal I}={\mathcal D}$ and fix $f\in C^{3+\alpha}_c(\Rd_+)$. Applying \eqref{grad_est_rd-1} with $f$ being replaced by ${\mathcal O}f\in C_c^{3+\alpha}(\Rd)$ and taking \eqref{Dir_eps} into account, we deduce that
\begin{equation}
|\nabla_x  G^\varepsilon_{\mathcal D}(t,s)f|^p\leq  2^{\ell_p+p-1}e^{C_p(t-s)}G^\varepsilon_{\mathcal N}(t,s)(|f|^p+|\nabla f|^p),
\label{epsilon}
\end{equation}
If $p\ge 2$, by Step 2 in the proof of Theorem \ref{existence}, we can let $\varepsilon\to 0^+$ in \eqref{epsilon} and obtain \eqref{GD_GN}
for functions in $C^{3+\alpha}_c(\Rd_+)$.
On the other hand, if $p\in (1,2)$, the function $|f|^p+|\nabla f|^p$ is not in $C^{2+\alpha}_c(\Rd)$. Anyway, 
$G^\varepsilon_{\mathcal N}(t,s)(|f|^p+|\nabla f|^p)$ still converges to $G_{\mathcal N}(t,s)(|f|^p+|\nabla f|^p)$ as $\varepsilon\to 0^+$. Indeed, we can approximate the function $|f|^p+|\nabla f|^p$ uniformly in $\Rd_+$ by a sequence of functions $g_n\in C^{2+\alpha}_c(\Rd_+)$.
Since
\begin{align*}
&\|G^\varepsilon_{\mathcal N}(t,s)(|f|^p+|\nabla f|^p)-G_{\mathcal N}(t,s)(|f|^p+|\nabla f|^p)\|_{L^{\infty}(K)}\\
\le & \|G^\varepsilon_{\mathcal N}(t,s)(|f|^p+|\nabla f|^p)-G^{\varepsilon}_{\mathcal N}(t,s)g_n\|_{L^{\infty}(K)}
+\|G^\varepsilon_{\mathcal N}(t,s)g_n-G_{\mathcal N}(t,s)g_n\|_{L^{\infty}(K)}\\
&+\|G_{\mathcal N}(t,s)(|f|^p+|\nabla f|^p)-G_{\mathcal N}(t,s)g_n\|_{L^{\infty}(K)},
\end{align*}
for any $n\in\N$ and any compact set $K\subset\Rd_+$, from \eqref{est_inf_eps} and \eqref{estinf}(i), we can estimate
\begin{align*}
&\|G^\varepsilon_{\mathcal N}(t,s)(|f|^p+|\nabla f|^p)-G_{\mathcal N}(t,s)(|f|^p+|\nabla f|^p)\|_{L^{\infty}(K)}\\
\le & 2\||f|^p+|\nabla f|^p-g_n\|_{\infty}
+\|G^\varepsilon_{\mathcal N}(t,s)g_n-G_{\mathcal N}(t,s)g_n\|_{L^{\infty}(K)}.
\end{align*}
Taking first the limsup as $\varepsilon\to 0^+$ and then the limit as $n\to +\infty$ in the previous inequality, the claim follows. Estimate \eqref{GD_GN} follows also in this case for functions in $C^{3+\alpha}_c(\Rd)$.

For a general function $f\in C^1_{\mathcal D}(\Rd_+)$ we consider a sequence $(f_n)\subset C^{3+\alpha}_c(\Rd_+)$ bounded in the $C^1_b$-norm and converging to $f$ in $C^1_{\rm loc}(\Rd_+)$. Writing \eqref{GD_GN} with $f$ being replaced by $f_n$, using Proposition \ref{nonesiste}(i) and letting $n\to +\infty$, we conclude the proof of the theorem when ${\mathcal I}=
{\mathcal D}$.

In the case when ${\mathcal I}={\mathcal N}$, replacing the function ${\mathcal O}f$ with ${\mathcal E}f$ and using the same arguments as above (taking \eqref{Neu_eps} into account), we can prove estimate \eqref{GD_GN} for any
function $f\in C^1_b(\overline{\Rd_+})$ with normal derivative vanishing on $\partial\Rd_+$. To extend \eqref{GD_GN} to any function $f\in C^1_b(\overline{\Rd_+})$, we consider the
sequence of functions $f_n={\mathcal E}f\star\rho_{1/n}$ ($n\in\N$) which belong to $C^{\infty}_b(\Rd)$ and have normal derivative on $\partial\Rd_+$ which identically vanishes.
We claim that $(f_n)$ is a bounded sequence in the $C^1_b$-norm which converges to $f$ in $C^1_{\rm loc}(\Rd_+)$ as $n\to +\infty$.
Indeed, since ${\mathcal E}f\in {\rm Lip}(\Rd)$, $f_n$ converges to $f$ uniformly in $\Rd_+$.
Moreover, $\nabla f_n=(\nabla{\mathcal E}f)\star\rho_{1/n}$. Hence, $\nabla f_n$ converges locally uniformly in $\Rd_+$ to $\nabla f$, and
$\|\nabla f_n\|_{C_b(\Rd_+)}\le \|\nabla f\|_{C_b(\Rd_+)}$ for any $n\in\N$. Hence, as above, Proposition \ref{nonesiste}(i) allows us to complete the proof.
\end{proof}

\begin{rmk}
{\rm
We stress that estimate \eqref{GD_GN}, with ${\mathcal I}={\mathcal D}$, can not hold with $C^1_{\mathcal D}(\Rd_+)$ being replaced by
$C^1_b(\overline{\Rd_+})$. Indeed, let $f\in C^1_b(\overline{\Rd_+})$ satisfy \eqref{GD_GN}. By the mean value theorem we deduce that
\begin{align}
|(G_{\mathcal D}(t,s)f)(x)-(G_{\mathcal D}(t,s)f)(y)|
\le & K\|f\|_{C^1_b(\overline{\Rd_+})}|x-y|,
\label{contradiction}
\end{align}
for any $x,y\in\overline{\Rd_+}$ and any $s,t\in I$, such that $s<t<s+1$, with $K=2^{\ell_p+p-1}e^{C_p}$.
Now, taking $x=(x',x_d)\in\Rd_+$ and $y=(x',0)$ in \eqref{contradiction}, we get
$|(G_{\mathcal D}(t,s)f)(x)|\le  K\|f\|_{C^1_b(\overline{\Rd_+})}x_d$.
Letting $t\to s^+$ in this inequality yields $|f(x)|\le K\|f\|_{C^1_b(\overline{\Rd_+})}x_d$, which shows that $f(x',0)=0$. Hence, $f\in C^1_{\mathcal D}(\Rd_+)$.
}
\end{rmk}

We now show that, when $c\equiv 0$, estimate \eqref{grad_est_rd} can be improved removing the dependence on $|f|^p$ from the right-hand side.

\begin{thm}
Let $c\equiv 0$ and $p\in (1,+\infty)$. Assume that Hypothesis $\ref{hyp2}(iii)$ is replaced by the following condition:
\begin{eqnarray*}
r(t,x)+\frac{k_1^2d^2}{4M_p}\eta(t,x)\leq K_p,\qquad\;\,(t,x)\in I\times\Rd_+,
\end{eqnarray*}
for some positive constant $K_p$.
Then, it holds that
\begin{equation}\label{GD_GN_0}
|(\nabla_x G_{\mathcal I}(t,s)f)(x)|^p\leq  e^{K_p(t-s)}(G_{\mathcal N}(t,s)|\nabla f|^p)(x),\qquad\;\,t>s,\;\,x\in\Rd_+,
\end{equation}
for any $f\in C^1_b(\overline{\Rd_+})$, when ${\mathcal I}={\mathcal N}$, and for any $f\in C_{\mathcal D}^1(\Rd_+)$, when ${\mathcal I}={\mathcal D}$.
\end{thm}

\begin{proof}
The claim can be proved arguing as in the proof of Theorem \ref{Dir_thm} replacing the function $w$ therein defined with the
function $w= (|\nabla_xG^{\varepsilon}_{{\mathcal N},n}(\cdot,s)f|^2+\tau)^{p/2}$,
where $\tau$ is a positive constant. Notice that $w$ has positive infimum in $(s,+\infty)\times\Rd$.
One can show that
$w_t-{\mathcal A}^{\varepsilon}(t)w\le pK_pw$
and deduce \eqref{grad_est_rd} with $C_p$ and $(|f|^p+|\nabla f|^p)$  being replaced, respectively, by $K_p$ and $(|\nabla f|^2+\tau)^{p/2}$.
Finally, letting $\varepsilon$ and $\tau$ tend to $0^+$ yields the assertion.
\end{proof}

In the following theorem, under stronger assumptions, we extend estimates \eqref{GD_GN} and \eqref{GD_GN_0} to the case
$p=1$.
\begin{thm}\label{p=1_prop}
Assume that the diffusion coefficients are independent of $x$ and Hypotheses $\ref{hyp2}(i)$-$(ii)$ are satisfied with
condition \eqref{cond-deriv-c}$(i)$ being replaced by
\begin{equation}\label{p=1}
\beta(t,x)\le\kappa\sqrt{|r(t,x)|},\qquad\;\,(t,x)\in I\times\Rd_+,
\end{equation}
for some positive constant $\kappa$. Then, estimate \eqref{GD_GN} holds true with $p=1$ and $C_1=\kappa^2$.
In particular, if $c\equiv 0$, then estimate \eqref{GD_GN_0} holds for $p=1$ with $K_1=-L_0\eta_0$.
\end{thm}
\begin{proof} The proof is similar to that of Theorem \ref{Dir_thm}, hence we just sketch it, pointing out the main differences.
The main step is the proof of the estimate
\begin{eqnarray*}
|\nabla_x G^\varepsilon (t,s)f|\leq 2e^{\kappa^2(t-s)}G^\varepsilon(t,s)(|f|+|\nabla f|),\qquad\;\,t>s\in I,
\end{eqnarray*}
in $\Rd$ for any $f\in C^1_b(\Rd)$ with positive infimum. To prove it, let
$w_1=[(u_n^{\varepsilon})^2+|\nabla_xu_n^{\varepsilon}|^2]^{1/2}$, where $u_n^{\varepsilon}=G_{{\mathcal N},n}^{\varepsilon}(\cdot,s)f$. We observe that, in this case
$D_t w_1-\A^\varepsilon(t)w_1=\psi_{1,1}+\psi_{2,1}+\psi_{3,1}$, where $\psi_{i,1}$ ($i=1,2,3$) are given by \eqref{psi-1}-\eqref{psi-3}.
Using \eqref{cond-deriv-c}(ii) to estimate $\psi_{1,1}$, and \eqref{colors} we obtain
\begin{align}\label{con_c}
D_t w_1-\A^\varepsilon(t)w_1
\le w_1^{-1}\big (r^\varepsilon|\nabla_x u_n^\varepsilon|^2+\beta^\varepsilon|u_n^\varepsilon||\nabla_x u_n^\varepsilon|\big ),
\end{align}
where $\beta^\varepsilon=({\mathcal E}\beta)^{\varepsilon}$. Since the H\"older inequality and \eqref{p=1} imply that
$\beta^\varepsilon\le \kappa\sqrt{|r^\varepsilon|}$, we conclude that
\begin{align*}
D_t w_1-\A^\varepsilon(t)w_1\le w_1^{-1}\big\{[(\kappa^{-2}(\beta^\varepsilon)^2+r^\varepsilon]|\nabla_x u_n^\varepsilon|^2+\kappa^2|u_n^\varepsilon|^2\big\}\le \kappa^2  w_1.
\end{align*}
Now, the proof of the first assertion can be completed arguing as in the proof of Theorem \ref{Dir_thm}, so that the details are omitted.

Finally, if $c\equiv 0$, \eqref{con_c} reduces to $D_t w_1-\A^\varepsilon(t)w_1\le w_1^{-1}r^\varepsilon|\nabla_x u_n^\varepsilon|^2\le -L_0\eta_0w_1$, by \eqref{dissipativity}, and the claim can be easily proved.
\end{proof}

\begin{rmk}
{\rm The estimate \eqref{GD_GN_0} holds with $p=1$ also when $c\equiv 0$ and the diffusion
coefficients depend on $x$, provided they satisfy the following conditions:
\begin{align*}
&D_iq_{jk}(t,x)+D_jq_{ik}(t,x)+D_kq_{ij}(t,x)=0,\\[1mm]
&\langle \nabla_xb(t,x)\xi,\xi\rangle +\frac{1}{2\eta_0}\sum_{i,j=1}^d\langle \nabla_xq_{ij}(t,x),\xi\rangle^2\le d_0|\xi|^2,
\end{align*}
for some $d_0\in \R$, any $(t,x)\in I\times\Rd_+$, any $\xi\in\Rd$ and any $i,j,k=1,\ldots,d$.
Indeed,
$D_i\tilde q_{jk}+ D_j\tilde  q_{ik}+D_k\tilde  q_{ij}\equiv 0$ in $I\times (\Rd\setminus\{0\})$, for any $i,j,k$ as above (see \eqref{qij-bj}). By convolution,
it is immediate to check that
$D_iq_{jk}^{\varepsilon}+ D_jq_{ik}^{\varepsilon}+D_kq_{ij}^{\varepsilon}\equiv 0$ in $I\times\Rd$, for any $i,j,k=1,\ldots,d$ and any $\varepsilon\in (0,1]$.
Similarly, arguing as in the proof of Proposition \ref{prop_cup} we deduce that
\begin{eqnarray*}
\langle \nabla_x\tilde b(t,x)\xi,\xi\rangle +\frac{1}{2\eta_0}\sum_{i,j=1}^d\langle \nabla_x\tilde q_{ij}(t,x),\xi\rangle^2\le d_0|\xi|^2,
\end{eqnarray*}
for any $(t,x)\in I\times (\Rd\setminus\{0\})$, and any $\xi\in\Rd$.
 Hence, by convolution and H\"older inequality, we obtain
\begin{eqnarray*}
\langle \nabla_xb^{\varepsilon}(t,x)\xi,\xi\rangle +\frac{1}{2\eta_0}\sum_{i,j=1}^d\langle \nabla_xq_{ij}^{\varepsilon}(t,x),\xi\rangle^2\le d_0|\xi|^2,
\end{eqnarray*}
for any $(t,x)\in I\times\Rd$, any $\xi\in\Rd$ and any $\varepsilon\in (0,1]$. We can thus apply \cite[Thm. 3.1]{Ang13OnT} to the operator $\A^{\varepsilon}(t)$,
which shows that
\begin{equation}\label{est_22}
|\nabla_x G^\varepsilon (t,s)f|\leq e^{d_0(t-s)}G^\varepsilon(t,s)|\nabla f|,\qquad\;\,t>s\in I,\;\,f\in C^1_b(\Rd).
\end{equation}
Writing \eqref{est_22} with $f$ being replaced with $\mathcal O f$ (resp. $\mathcal E f$) and letting $\varepsilon\to 0$ leads to \eqref{GD_GN_0} with $p=1$, $K_1=d_0$ and $\mathcal{J}=\mathcal{D}$ (resp. $\mathcal{J}=\mathcal{N}$).
}
\end{rmk}

As a consequence of Theorems \ref{Dir_thm}, \ref{p=1_prop} and estimate \eqref{estinf} we deduce the following uniform gradient estimates.

\begin{coro}\label{coro_uniform}
For ${\mathcal I}\in\{{\mathcal D},{\mathcal N}\}$ the uniform gradient estimate
\begin{eqnarray}\label{GD_GN-5}
\|\nabla_x G_{\mathcal I}(t,s)f\|_{\infty}\leq  2e^{\frac{C_2-c_0}{2}(t-s)}\|f\|_{C^1_b(\overline{\Rd_+})},\qquad\;\,t>s\in I,
\end{eqnarray}
holds for any $f\in C^1_{\mathcal D}(\Rd_+)$, when ${\mathcal I}={\mathcal D}$, and for any $f\in C^1_b(\overline{\Rd_+})$, when ${\mathcal I}={\mathcal N}$.

Assume in addition, the diffusion coefficients are independent of $x$.
If condition \eqref{p=1} is satisfied, then \eqref{GD_GN-5} holds with $2e^{(C_2-c_0)(t-s)/2}$ being replaced with $e^{(\kappa^2-c_0)(t-s)}$.
On the other hand, if $c\equiv 0$, then \eqref{GD_GN-5} is satisfied with $2e^{(C_2-c_0)(t-s)/2}$ and $\|f\|_{C^1_b(\overline{\Rd_+})}$ being replaced with $e^{-(L_0\eta_0+c_0)(t-s)}$ and $\|\nabla f\|_{\infty}$, respectively.
\end{coro}

\subsection{$C^0$-$C^1$ uniform and pointwise estimates}

We now prove a second type of pointwise gradient estimates which, besides the interest in their own, will be used in Section \ref{sect-4}
to study the asymptotic behaviour of $G_{\mathcal D}(t,s)f$ and $G_{\mathcal N}(t,s)f$ as $t\to +\infty$.

\begin{thm}\label{thm_3.11}
For ${\mathcal I}\in \{{\mathcal D}, {\mathcal N}\}$, every $p\in(1,+\infty)$ and $s\in I$  the gradient estimate
\begin{equation}\label{point_p_D}
|\nabla_x G_{\mathcal I}(t,s)f|^p\leq c_p\,e^{\omega_p (t-s)}(t-s)^{-\frac{p}{2}}G_{\mathcal N}(t,s)|f|^p,\qquad\;\,t>s\in I,
\end{equation}
holds in $\Rd_+$, for any $f\in C_b(\Rd_+)$ when ${\mathcal I}={\mathcal D}$, and
 for any $f\in C_b(\overline{\Rd_+})$ when ${\mathcal I}={\mathcal N}$.
 Here, $c_p$ is a positive constant and $\omega_p$ is any constant larger than $\min\{C_p,0\}$, where $C_p$ is given by \eqref{cond-r-eta-c}.
As a consequence, the following uniform gradient estimate
\begin{eqnarray*}
\|\nabla_x G_{\mathcal I}(t,s)f\|_{\infty}\leq \sqrt{c_2}\,e^{\frac{\omega_2-c_0}{2}(t-s)}(t-s)^{-\frac{1}{2}}\|f\|_{\infty},
\end{eqnarray*}
is satisfied for any $t>s$ and any $f$ as above.
\end{thm}

\begin{proof}
It suffices to prove \eqref{point_p_D} with $G_{\mathcal I}(t,s)$ being replaced by $G^{\varepsilon}(t,s)$ and $f\in C_c^{3+\alpha}(\Rd)$ with a positive infimum since, then, the same arguments as in the proof of Theorem \ref{Dir_thm} will allow us to conclude.
Hence, let us prove that
\begin{equation}
|\nabla_x G^{\varepsilon}(t,s)f|^p\leq c_p\,e^{\omega_p (t-s)}(t-s)^{-\frac{p}{2}}G^{\varepsilon}(t,s)|f|^p,
\label{stima-p}
\end{equation}
for such $f$'s and any $\varepsilon\in (0,1]$. Note that the case $p>2$ follows from the case $p=2$, writing
$|\nabla_xG^{\varepsilon}(t,s)f|^p
\le \left (c_2\,e^{\omega_2 (t-s)}(t-s)^{-1}G^\varepsilon(t,s)|f|^2\right )^{p/2}$
and using \eqref{kernerl_eps-bis} to estimate the right-hand side of the previous inequality.

For $p\in (1,2]$ and $f$ as above, the claim can be proved adapting the arguments in the proof of \cite[Prop. 3.3]{LorZam09Cor}, which are based on the gradient estimate
\eqref{grad_est_rd}. For the reader's convenience we provide the ideas of the proof.

We introduce the function $g=G_{{\mathcal N},n}^\varepsilon(t,\cdot)|G_{{\mathcal N},n}^\varepsilon(\cdot, s)f|^p$,
where $G_{{\mathcal N},n}^\varepsilon(t,s)$ is the evolution operator introduced in \eqref{grad_est_rd-0}. This function is differentiable in $(s,t)$ (see \cite[Thm. 2.3(ix)]{Acq88Evo}) and
\begin{align}\label{un attimo}
g'(\sigma)
& =G_{{\mathcal N},n}^\varepsilon(t,\sigma)\Big[-p(p-1)|G_{{\mathcal N},n}^\varepsilon(\sigma,s)f|^{p-2}|\sqrt{Q(\sigma,\cdot)}
\nabla_xG_{{\mathcal N},n}^\varepsilon(\sigma,s)f|^2\notag\\
&\qquad\qquad\qquad\;\,+(1-p)c\,|G_{{\mathcal N},n}^\varepsilon(\sigma, s)f|^p\Big]\nonumber\\
& \le -p(p-1)\eta_0 G_{{\mathcal N},n}^\varepsilon(t,\sigma)\Big[|G_{{\mathcal N},n}^\varepsilon(\sigma,s)f|^{p-2}|\nabla_xG_{{\mathcal N},n}^\varepsilon(\sigma,s)f|^2\Big].
\end{align}
Integrating the first and the last sides of \eqref{un attimo} with respect to $\sigma$ in $[s+\delta,t-\delta]$ and then letting $n$ and $\delta$ tend to $+\infty$ and $0$, respectively, we get
\begin{equation}\label{antonio}
G^\varepsilon(t,s)|f|^p\ge p(p-1)\eta_0 \int_{s}^{t}G^\varepsilon(t,\sigma)\Big[|G^\varepsilon(\sigma,s)f|^{p-2}|\nabla_xG^\varepsilon(\sigma,s)f|^2\Big]d\sigma.
\end{equation}
Thanks to \eqref{grad_est_rd}, we can estimate
\begin{align}\label{first}
|\nabla_x G^\varepsilon(t,s)f|^p&=|\nabla_x G^\varepsilon(t,\sigma)G^\varepsilon(\sigma,s)f|^p\nonumber\\
&\le 2^{\ell_p+p-1}e^{C_p(t-\sigma)}[G^\varepsilon(t,\sigma)|G^\varepsilon(\sigma,s)f|^p+G^\varepsilon(t,\sigma)|\nabla_x G^\varepsilon(\sigma,s)f|^p]
\notag\\
&\le 2^{\ell_p+p-1}e^{C_p(t-\sigma)}[I_1(\sigma)+I_2(\sigma)].
\end{align}
Applying estimate \eqref{kernerl_eps-bis} with $r=2/p$,
$\psi_1=|G^\varepsilon(\sigma,s)f|^{p(p-2)/2}|\nabla_x G^\varepsilon(\sigma,s)f|^p$ and $\psi_2=
|G^\varepsilon(\sigma,s)f|^{p(2-p)/2}$,
and taking into account that $I_1(\sigma)\le G^{\varepsilon}(t,s)|f|^p$ for any $\sigma\in (s,t)$,
we get
\begin{align*}
I_2(\sigma)\le \frac{p}{2\,}\gamma^{\frac{2}{p}}\,G^\varepsilon(t,\sigma)\left(|G^\varepsilon(\sigma,s)f|^{p-2}|\nabla_x G^\varepsilon(\sigma,s)f|^2\right)+ \left(1-\frac{p}{2}\right)\gamma^{\frac{2}{p-2}}\, G^\varepsilon(t,s)|f|^p,
\end{align*}
for any $\gamma>0$ and any $\sigma\in (s,t)$. Thus,
estimate \eqref{first} becomes
\begin{align}\label{second11}
|\nabla_x G^\varepsilon(t,s)f|^p \le 2^{\ell_p+p-1}e^{C_p(t-\sigma)}\Big[&\frac{p}{2\,}\gamma^{\frac{2}{p}}\,G^\varepsilon(t,\sigma)\left(|G^\varepsilon(\sigma,s)f|^{p-2}|\nabla_x G^\varepsilon(\sigma,s)f|^2\right)\nonumber\\
&+ \left(\Big(1-\frac{p}{2}\Big)\gamma^{\frac{2}{p-2}}+1\right)\, G^\varepsilon(t,s)|f|^p\Big].
\end{align}
Now we multiply both sides of \eqref{second11} by $e^{-C_p(t-\sigma)}$ and we integrate the so obtained inequality with respect to $\sigma\in (s,t)$ taking \eqref{antonio} into account. Finally, minimizing with respect to $\gamma>0$ yields
\begin{eqnarray*}
|\nabla_x G^\varepsilon(t,s)f|^p\le  \frac{\tau_p\,C_p}{1-e^{-C_p(t-s)}} (t-s)^{1-\frac{p}{2}}[1+(t-s)^{\frac{p}{2}}]G^\varepsilon(t,s)|f|^p,
\end{eqnarray*}
for some positive constant $\tau_p$. Thus, estimate \eqref{stima-p} follows.
\end{proof}

\section{Evolution systems of measures and asymptotic behaviour}
\label{sect-4}

In this section we prove the existence of an evolution system of measures $(\mu^{\mathcal N}_t)_{t\in I}$ associated to the Neumann evolution operator $G_{\mathcal N}(t,s)$ which turns out to be subinvariant for the Dirichlet evolution operator $G_{\mathcal D}(t,s)$.
We study the asymptotic behaviour of both the evolution operators $G_{\mathcal D}(t,s)$ and $G_{\mathcal N}(t,s)$ in $L^p$-spaces with respect to the evolution system of measures  $\{\mu^{\mathcal N}_t\}_{t\in I}$. We also deduce logarithmic Sobolev inequalities with respect the measures $(\mu^{\mathcal N}_t)_{t\in I}$ and the hypercontractivity property for the evolution operators $G_{\mathcal D}(t,s)$ and $G_{\mathcal N}(t,s)$.

To begin with, let us recall the following definition.

\begin{defi}
\label{def-4.1}
Let $\mathcal O\subset \Rd$ be an open set or the closure of an open set. Further, let $U(t,s)$ be an evolution operator on $C_b(\mathcal O)$.
A family $(\mu_t)_{t\in I}$ of probability measures on $\mathcal O$ is called
\begin{enumerate}[\rm (i)]
\item
an evolution system of measures for $U(t,s)$, if
\begin{eqnarray*}
\int_{\mathcal O}U(t,s)f\, d\mu_t=\int_{\mathcal O}f\,d\mu_s,
\end{eqnarray*}
for every $f\in C_b(\mathcal O)$ and every $t>s\in I$;
\item
an evolution system of subinvariant measures for $U(t,s)$ if,
\begin{eqnarray*}
\int_{\mathcal O}U(t,s)f\,d\mu_t\le \int_{\mathcal O}f\,d\mu_s,
\end{eqnarray*}
for every nonnegative $f\in C_b(\mathcal O)$ and every $t>s\in I$.
\end{enumerate}
\end{defi}

\begin{rmk}\label{c=0}{\rm
By virtue of Theorem \ref{existence},  $0\le G_{\mathcal I}(t,s)\one\le \one$ for any
$I\ni s<t$, where ${\mathcal I}$ is either ${\mathcal D}$ or ${\mathcal N}$. If $\{\mu_t\}_{t\in I}$ is an evolution
system of measures for $G_{\mathcal I}(t,s)$, then $\langle\mu_t,G_{\mathcal I}(t,s)\one\rangle=1$
for any $I\ni s<t$. Hence, $G_{\mathcal I}(t,s)\one=\one$ $\mu_t$-a.e. in $\Rd_+$.
We claim that  $G_{\mathcal I}(t,s)\one=\one$ everywhere in $\Rd_+$. Indeed, by \eqref{estinf}, $G_{\mathcal I}(t,s)\one\le\one$ in $\Rd_+$ for any $t>s$. Let $x_0\in\Rd_+$ be
such that $(G_{\mathcal I}(t,s)\one)(x_0)=1$. Then, $(t,x_0)$ is a maximum point of the function $G_{\mathcal I}(\cdot,s)\one$.
The classical maximum principle (see e.g., \cite[Thm 3.3.5]{protter}) shows that $(G_{\mathcal I}(t,s)\one)(x)=1$ for any $x\in\Rd_+$.

Clearly, the equality $G_{\mathcal I}(t,s)\one\equiv\one$ can not hold if ${\mathcal I}={\mathcal D}$ since
$G_{\mathcal D}(t,s)\one$ vanishes on $\partial\Rd_+$. This shows that there exist no evolution systems of measures for the Dirichlet evolution operator $G_{\mathcal D}(t,s)$.
At the same time the equality $G_{\mathcal N}(t,s)\one=\one$ implies that $c\equiv 0$ since $(D_t-\A(t))\one=c(t,\cdot)\one$.
}\end{rmk}

\begin{lemm}
\label{lemma-jr}
Any evolution system of measures $\{\mu_s\}_{s\in I}$ for the evolution operator $G_{\mathcal N}(t,s)$ is an evolution system of subinvariant measures for
the operator $G_{\mathcal D}(t,s)$. Moreover, for any $I\ni s<t$, the operators $G_{\mathcal D}(t,s)$ and $G_{\mathcal N}(t,s)$ extend to contractions from $L^p(\Rd_+,\mu_s)$ into $L^p(\Rd_+,\mu_t)$ for any $p\in [1,+\infty)$.
\end{lemm}

\begin{proof}
The first assertion follows immediately from Proposition \ref{nonesiste}(ii).
As far as the other claim is concerned, we recall that, for any $f\in C^{\infty}_c(\Rd_+)$ and any $I\ni s<t$,
$G_{\mathcal D}(t,s)f$ is the pointwise limit, as $\varepsilon\to 0^+$, of the family of functions
$G^{\varepsilon}(t,s){\mathcal O}f$ (see Step 2 of the proof of Theorem \ref{existence}). From \eqref{kernerl_eps-bis} it follows that
$|G_{\mathcal D}(t,s)f|^p\le G_{\mathcal N}(t,s)|f|^p$. Therefore, using the subinvariance
of the measures $\{\mu_t\}_{t\in I}$, we get
\begin{eqnarray*}
\langle\mu_t,|G_{\mathcal D}(t,s)f|^p\rangle\le  \langle\mu_t,G_{\mathcal N}(t,s)|f|^p\rangle=\langle\mu_s,|f|^p\rangle.
\end{eqnarray*}
Since $C^{\infty}_c(\Rd_+)$ is dense in $L^p(\Rd_+,\mu_t)$ for any $t\in I$,
$G_{\mathcal D}(t,s)$ can be extended to a contraction from $L^p(\Rd_+,\mu_s)$ to $L^p(\Rd_+,\mu_t)$.

In the same way one can prove that also $G_{\mathcal N}(t,s)$ extends to a contraction from $L^p(\Rd_+,\mu_s)$ to $L^p(\Rd_+,\mu_t)$.
\end{proof}

In view of Remark \ref{c=0} and Lemma \ref{lemma-jr},
in the rest of this section we assume the following set of assumptions.

\begin{hyp}\label{hyp3}
\begin{enumerate}[\rm (i)]
\item
Hypotheses $\ref{hyp1}(i)$, $(iii)$-$(vi)$ are satisfied;
\item
the coefficients $q_{ii}(\cdot,0)$ and $b_i(\cdot,0)$ $(i=1,\ldots,d)$ are bounded in $[s,+\infty)$ for any $s\in I$.
\item
$c\equiv 0$;
\item
for any $p\in (1,+\infty)$ there exists a positive constant $C_p$ such that
\begin{eqnarray*}
r(t,x)+\left (\frac{k_1^2d^2}{4M_p}-M_p\right )\eta(t,x)\le C_p,\qquad\;\,(t,x)\in I\times\Rd_+,
\end{eqnarray*}
where $M_p=\min\{1,p-1\}$.
\end{enumerate}
\end{hyp}

We find it convenient to prove the existence of an evolution system of measures for the evolution operator $G_{\mathcal N}(t,s)$ as the weak limit of some evolution systems of measures associated to
the evolution operator $G^{\varepsilon}(t,s)$ in the whole of $\Rd$.

\begin{lemm}\label{lemma_complete}
For any $\varepsilon\in (0,1]$ there exists a unique evolution system of measures $(\mu_t^\varepsilon)$ associated to $G^\varepsilon(t,s)$
such that the system $\{\mu_t^{\varepsilon}\}_{t\ge t_0}$ is tight for any $t_0\in I$. Such a system is uniformly tight with respect to $\varepsilon\in (0,1]$.
Moreover, for any $t\in I$, each measure $\mu_t^\varepsilon$ is absolutely continuous with respect the Lebesgue measure. More precisely, there exists a locally H\"older continuous function $\rho_\varepsilon:I\times\Rd\to \R$, which is even with respect to the variable $x_d$, and $\mu_t^\varepsilon(dx)=\rho_\varepsilon(t,x)dx$ for any $t\in I$ and any $\varepsilon\in (0,1]$.
\end{lemm}
\begin{proof}
The existence of an evolution system of measures for the evolution operator $G^{\varepsilon}(t,s)$ and the absolute continuity of each measure $\mu_t^{\varepsilon}$ with respect to the Lebesgue measure follow from \cite[Prop. 5.2 \& Thm. 5.4]{KunLorLun09Non}, which require the existence of a function $\varphi\in C^2(\Rd)$ diverging to $+\infty$ as $|x|\to +\infty$
such that, for any $s\in I$, $\limsup_{|x|\to +\infty}(\A^{\varepsilon}(t)\varphi(x))/\varphi(x)\le -c$ for some positive constant $c$ and any $t>s$. In our situation, the function $\varphi$, defined by $\varphi(x)=1+|x|^2$ for any $x\in \Rd$, has the previous property for any $t\in I$, with the
constant $c$ being independent of $\varepsilon$. Moreover, $\varphi$ is integrable with respect to the measure $\mu_t^{\varepsilon}$ for any $t\in I$ and any $\varepsilon\in (0,1]$ and, for any $s\in I$,
\begin{equation}
H_{s,1}:=\sup\{\langle\mu_t^{\varepsilon},\varphi\rangle:\ t\ge s,\,\varepsilon\in (0,1]\}<+\infty.
\label{Hs}
\end{equation}
This fact and Chebyshev inequality imply that the family of measures $\{\mu_t^{\varepsilon}\}_{t\ge s}$ is tight for any $s\in I$, uniformly with respect to $\varepsilon\in (0,1]$. More precisely,
\begin{equation}
\mu_t^{\varepsilon}(\Rd\setminus B_r)\le \frac{H_{s,1}}{r^2},\qquad\;\,t>s,\;\,\varepsilon\in (0,1].
\label{tight-epsilon}
\end{equation}

In view of \cite[Thm. 3.8]{BogKryRoc01} and \cite[Lemma 2.4]{AngLorLun}, $\mu_t^{\varepsilon}=\rho_{\varepsilon}(t,\cdot)dx$ for any $t\in I$, where the function $\rho_{\varepsilon}$ is locally $\gamma$-H\"older continuous in $I\times\Rd$ for any $\gamma\in (0,1)$.
To prove that $\rho_{\varepsilon}(t,\cdot)$ is even with respect to the variable $x_d$ for any $t\in I$, we need to recall briefly the construction of the tight
evolution system of measures in \cite[Thm. 5.4]{KunLorLun09Non}.
Let $n_0\in\mathbb Z$ be the smallest integer in $I$ and fix $x\in\Rd$. The family of measures $\{\mu_{t,s,x}^{\varepsilon}\}_{t>s\ge n_0}$ where
\begin{eqnarray*}
\mu_{t,s,x}^{\varepsilon}(A)=\frac{1}{t-s}\int_s^t(G^{\varepsilon}(\tau,s)\one_A)(x)d\tau,\qquad\;\,t>s\in I,\;\,A\in {\mathcal B}(\Rd),
\end{eqnarray*}
is tight. Prokhorov theorem and a diagonal argument guarantee the existence of a sequence $(t_k^{\varepsilon})$ diverging to $+\infty$ such that
$\mu_{t_k,n,x}^{\varepsilon}$ weakly$^*$ converges to a measure $\mu_{n,x}^{\varepsilon}$ as $k\to +\infty$ for any $n\in\mathbb Z$ such that $n\ge n_0$.
For $t\in I\setminus\{n\in\mathbb Z: n\ge n_0\}$, one defines the measure $\mu_{t,x}^{\varepsilon}$ by setting $\mu_{t,x}^{\varepsilon}=(G^{\varepsilon}(n,t))^*\mu_{n,x}^{\varepsilon}$, where $(G^{\varepsilon}(n,t))^*$ is the operator adjoint to $G^{\varepsilon}(n,t)$. The family $\{\mu_{t,x}^{\varepsilon}\}_{t\in I}$ is a tight evolution system of measures for the evolution operator $G^{\varepsilon}(t,s)$. Clearly, the construction of the previous system of measures {\it apriori} depends on the choice of $x\in\Rd$ as well as on the choice of the sequence $(t_k^{\varepsilon})$, but the arguments in \cite[Rem. 2.8]{AngLorLun} show that the tight evolution system of measures is unique whenever a gradient estimate of the type
$\|\nabla_x G^{\varepsilon}(t,s)f\|_{\infty}\le Ce^{-\sigma (t-s)}\|\nabla f\|_{\infty}$
is satisfied for any $f\in C^1_b(\Rd)$, any $t>s\in I$, and some $C, \sigma>0$, which is actually our case in view of Proposition \ref{p=1_prop}.

So, let us fix $x_0=(0,\ldots,0,1)$ and $x_1=-x_0$. Further, let $(t_k^{\varepsilon})$ be a sequence
diverging to $+\infty$ as $k\to +\infty$  such that both $\mu_{t_k,n,x_0}^{\varepsilon}$ and $\mu_{t_k,n,x_1}^{\varepsilon}$ weakly$^*$ converge to $\mu_n^{\varepsilon}$ for any $n\ge n_0$.

Fix a function $f\in C_b(\Rd)$, odd with respect to the $x_d$-variable. Then, as Step 1 in the proof of Theorem \ref{Dir_thm} shows, the function $G^{\varepsilon}(t,s)f$ is
odd with respect to the $x_d$-variable as well. Therefore,
\begin{align*}
\int_{\R^d}f\,d\mu_{t_k,n,x_1}^{\varepsilon}=&
\frac{1}{t_k-n}\int_n^{t_k}(G^{\varepsilon}(\tau,n)f)(x_1)d\tau
=-\frac{1}{t_k-n}\int_n^{t_k}(G^{\varepsilon}(\tau,n)f)(x_0)d\tau\\
=&-\int_{\R^d}f\,d\mu_{t_k,n,x_0}^{\varepsilon}.
\end{align*}
Letting $k\to +\infty$ gives
$\langle\mu_n^{\varepsilon},f\rangle=0$ for any $n\ge n_0$.
Using the definition of the evolution system of measures, we can extend the previous formula to any $t\in I$. 
Now a standard argument allows us to conclude that $\rho_{\varepsilon}(t,\cdot)$ is even with respect to the last variable: we fix $\psi\in C_b(\R^d)$ and write $\langle\mu_t,\overline\psi\rangle=0$ with $f=\overline\psi$, where $\overline\psi(x):=\psi(x_1,\ldots,x_{d-1},x_d)
-\psi(x_1,\ldots,x_{d-1},-x_d)$ for any $x\in\Rd$. Using a straightforward change of variables and the arbitrariness of $\psi$, the assertion follows at once.
\end{proof}

\begin{lemm}
\label{lemma_4.4}
There exist an infinitesimal sequence $(\varepsilon_n)$ and an evolution system of measures $\{\mu_t^{\mathcal N}\}_{t\in I}$ for the operator $G_{\mathcal N}(t,s)$ such that, for any bounded sequence $(f_n)\in C_b(\overline{\Rd_+})$ converging locally uniformly to $f$ as $n \to +\infty$,
\begin{equation}
\lim_{n\to +\infty}\langle \tilde\mu_t^{\varepsilon_n},f_n\rangle=\langle\mu_t^{\mathcal N},f\rangle,\qquad\;\,t\in I,
\label{limit-doppio}
\end{equation}
where $\tilde\mu^{\varepsilon}_t=2\rho_{\varepsilon}(t,\cdot)dx$. Moreover, for any $s\in I$, the system $\{\mu_t^{\mathcal N}\}_{t\ge s}$ is tight.
\end{lemm}

\begin{proof}
Denote as usually by $G_{\mathcal N}^{\varepsilon}(t,s)$ the evolution operator associated with the operator $\A^{\varepsilon}(t)$ (see \eqref{A-varepsilon}) with homogeneous Neumann boundary conditions on $\partial\Rd_+$. Since, for any $f\in C_b(\overline{\Rd_+})$,
$G^{\varepsilon}(t,s){\mathcal E}f$ and $\rho_{\varepsilon}$ are even with respect to the $x_d$-variable, taking \eqref{Neu_eps} into account, we can infer that
\begin{equation}
\langle \tilde\mu_t^{\varepsilon}, G_{\mathcal N}^{\varepsilon}(t,s)f\rangle= \langle \tilde\mu_s^{\varepsilon},f\rangle,
\label{inv-eps}
\end{equation}
and each $\tilde\mu_r^{\varepsilon}$ ($r\in I$) is a probability measure in $\Rd_+$.

Using the tightness of the family of measures $\{\tilde\mu_t^{\varepsilon}\}_{t>n_0, \varepsilon\in (0,1]}$, where $n_0$ is the smallest integer in $I$, the same procedure as in the proof of Lemma \ref{lemma_complete} shows that there
 exist an infinitesimal sequence $(\varepsilon_k)$ and probability measures $\mu_n^{\mathcal N}$ ($n\in\N$, $n\ge n_0$) such that
$\tilde\mu_n^{\varepsilon_k}$ weakly$^*$ converges to $\mu_n^{\mathcal N}$ as $k\to +\infty$. We claim that, for any $t,s\in\mathbb Z$ with $t>s\ge n_0$, we have that
\begin{equation}
\langle \mu_t^{\mathcal N},G_{\mathcal N}(t,s)f\rangle = \langle \mu_s^{\mathcal N},f\rangle,
\label{inv-eps-1}
\end{equation}
or, equivalently, $\mu_s^{\mathcal N}=(G_{\mathcal N}(t,s))^*\mu_t^{\mathcal N}$. Formula \eqref{inv-eps-1}
follows writing \eqref{inv-eps} with $\varepsilon_k$ replacing $\varepsilon$ and letting $k\to +\infty$. Clearly, the right-hand side of \eqref{inv-eps} converges to the right-hand side of \eqref{inv-eps-1}.
As far as the convergence of the left-hand side is concerned, we observe that \eqref{tight-epsilon} implies that
\begin{align*}
\langle \tilde\mu_t^{\varepsilon_k},|G_{\mathcal N}^{\varepsilon_k}(t,s)f\hskip  -1truemm-\hskip -1truemm G_{\mathcal N}(t,s)f|\rangle\!
=&\langle \tilde\mu_t^{\varepsilon_k},|G_{\mathcal N}^{\varepsilon_k}(t,s)f-G_{\mathcal N}(t,s)f|\chi_{B_r^+}\rangle\\
&+\langle \tilde\mu_t^{\varepsilon_k},|G_{\mathcal N}^{\varepsilon_k}(t,s)f-G_{\mathcal N}(t,s)f|\chi_{\Rd_+\setminus B_r^+}\rangle\\
\le &\|G_{\mathcal N}^{\varepsilon_k}(t,s)f\hskip -.5truemm - \hskip -.5truemm G_{\mathcal N}(t,s)f\|_{C(\overline{B_r^+})}\hskip -1truemm+\hskip -1truemm 2\|f\|_{\infty}\tilde\mu_t^{\varepsilon_k}(\Rd_+\setminus B_r^+)\\
\le &\|G_{\mathcal N}^{\varepsilon_k}(t,s)f-G_{\mathcal N}(t,s)f\|_{C(\overline{B_r^+})}+\frac{2H_{s,1}}{r^2}\|f\|_{\infty}.
\end{align*}
Therefore,
\begin{align}
&|\langle \tilde\mu_t^{\varepsilon_k},G_{\mathcal N}^{\varepsilon_k}(t,s)f\rangle-\langle \mu_t^{\mathcal N},G_{\mathcal N}(t,s)f\rangle |\notag\\[1mm]
\le &\langle \tilde\mu_t^{\varepsilon_k},|G_{\mathcal N}^{\varepsilon_k}(t,s)f-G_{\mathcal N}(t,s)f|\rangle
+|\langle \tilde\mu_t^{\varepsilon_k},G_{\mathcal N}(t,s)f\rangle-\langle\mu_t^{\mathcal N},G_{\mathcal N}(t,s)f\rangle |\notag\\[1mm]
\le & \|G_{\mathcal N}^{\varepsilon_k}(t,s)f-G_{\mathcal N}(t,s)f\|_{C(\overline{B_r^+})}\hskip -.5truemm +\hskip -.5truemm \frac{2H_{s,1}}{r^2}\|f\|_{\infty}
\hskip -.5truemm +\hskip -.5truemm |\langle \tilde\mu_t^{\varepsilon_k},G_{\mathcal N}(t,s)f\rangle\hskip -.5truemm - \hskip -.5truemm\langle \mu_t^{\mathcal N},G_{\mathcal N}(t,s)f\rangle |
\label{formula-inv}
\end{align}
and the last side of the previous chain of inequalities vanishes letting first $k$ and then $r$ tend to $+\infty$. Hence,
\eqref{inv-eps-1} follows.

Now, for $t\in I\setminus\mathbb Z$, we set
$\mu_t^{\mathcal N}:=(G_{\mathcal N}(n,t))^*\mu_n^{\mathcal N}$, where $n=[t]+1$. Clearly, $\{\mu_t^{\mathcal N}\}_{t\in I}$ is an evolution system of measures for the operator $G_{\mathcal N}(t,s)$. Moreover, $\tilde\mu_t^{\varepsilon_n}$ weakly$^*$ converges to $\mu_t^{\mathcal N}$ as $n\to +\infty$,
for any $t\in I$. Indeed, for any $n\in\mathbb Z$ such that $n>t$ and any $f\in C_b(\overline{\Rd_+})$, from \eqref{formula-inv} it follows that
\begin{eqnarray*}
\langle\mu_t^{\mathcal N},f\rangle =\langle \mu_n^{\mathcal N},G_{\mathcal N}(n,t)f\rangle=
\lim_{k\to +\infty}\langle \tilde\mu_n^{\varepsilon_k},G^{\varepsilon_k}_{\mathcal N}(n,t)f\rangle
=\lim_{k\to +\infty}\langle \tilde\mu_t^{\varepsilon_k},f\rangle.
\end{eqnarray*}

Further, observe that, for any $t_0\in I$, $\{\mu_t^{\mathcal N}\}_{t\ge t_0}$ is a tight system. Indeed, \eqref{Hs} shows that
$\langle\tilde\mu_t^{\varepsilon_k},\varphi\rangle\le H_{t_0,1}$ for any $t\ge t_0$ and any $k\in\R$, where $\varphi(y)=1+|y|^2$ for any $y\in\Rd$.
By monotonicity, $\langle\tilde\mu_t^{\varepsilon_k},\varphi\wedge m\rangle\le H_{t_0,1}$ for any $m\in\N$ and any $t,k$ as above.
Letting first $k$ and then $m$ tend to $+\infty$ we deduce that $\langle\mu_t^{\mathcal N},\varphi\rangle\le H_{t_0,1}$ for any $t\ge t_0$ and, then,
by Chebyshev inequality, the system $\{\mu_t^{\mathcal N}\}_{t\ge t_0}$ is tight.

Finally, we observe that formula \eqref{limit-doppio} has been essentially already proved. Indeed, the same argument used to obtain \eqref{formula-inv}
yields \eqref{limit-doppio} for any $t\in I$.
\end{proof}

\begin{rmk}
\label{rem-4.7}
{\rm
Note that, for any $p>1$ and any $s\in I$, the function $\varphi_p(x)=1+|x|^{2p}$
satisfies the condition $\limsup_{|x|\to +\infty}({\mathcal A}^{\varepsilon}(t)\varphi_p(x))/\varphi_p(x)<-c_s$ for some positive constant $c_s$, any $t\in [s,+\infty)$ and any
$\varepsilon\in (0,1]$. Therefore, the same arguments used in the proof of Lemmas \ref{lemma_complete} and \ref{lemma_4.4} show that
each function $\varphi_p$ is integrable with respect to the measure
$\mu_t^{\mathcal N}$ for any $t\in I$, and there exists a positive constant $H_{s,p}$ such that
\begin{equation}
\langle\mu_t^{\mathcal N},\varphi_p\rangle\le H_{s,p},\qquad\;\,t\ge s\in I.
\label{Hsp}
\end{equation}
}
\end{rmk}
\medskip
\par

The gradient estimates in the previous section show that each operator $G_{\mathcal I}(t,s)$ is bounded from $L^p(\Rd_+,\mu_s^{\mathcal N})$ into
the Sobolev space $W^{1,p}(\Rd_+,\mu_t^{\mathcal N})$.

\begin{prop}
The family $\{\mu_t^{\mathcal N}\}_{t\in I}$ is an evolution system of subinvariant measures for the evolution operator $G_{\mathcal D}(t,s)$. Moreover,
for any $p\in (1,+\infty)$, any $t>s\in I$ and ${\mathcal I}\in \{{\mathcal D}, {\mathcal N}\}$ it holds that
\begin{equation}\label{nor_p_fun}
\|\nabla_x G_{\mathcal I}(t,s)f\|_{L^p(\Rd_+,\mu_t^{\mathcal N})}\le c_p\,e^{\omega_p (t-s)}(t-s)^{-\frac{p}{2}}\|f\|_{L^p(\Rd_+,\mu_s^{\mathcal N})},
\end{equation}
for any $f\in L^p(\Rd_+,\mu_s^{\mathcal N})$, where $c_p$ and $\omega_p$ are the constants in Theorem $\ref{thm_3.11}$.
\end{prop}
\begin{proof}
The first part of the statement follows from Lemma \ref{lemma-jr}.
As far as the second part of the statement is concerned, we observe that, since $\{\mu_t^{\mathcal N}\}_{t\in I}$ is an evolution system of measures (resp. of subinvariant measures) for $G_{\mathcal N}(t,s)$
(resp. for $G_{\mathcal D}(t,s)$), we get formula \eqref{nor_p_fun} as consequence of \eqref{point_p_D} and of the density of the space $C^\infty_c(\Rd_+)$ in $L^p(\Rd_+,\mu_s)$ for any $s\in I$.
\end{proof}

\begin{rmk}
{\rm Let $\A(t)=\sum_{i,j=1}^d q_{ij}(t,\cdot)D_{ij}+\sum_{i=1}^d b_i(t,\cdot) D_i-c(t,\cdot)$ with $c$ nonnegative and such that the diffusion and drift coefficients satisfy Hypotheses \ref{hyp3}. Let
$\{\mu_t^{\mathcal N}\}_{t\in I}$ be the evolution system of invariant measures for the evolution operator $G_{\mathcal N}(t,s)$, associated with the operator $\A(t)+c(t,\cdot)$. Then, such a system of measures
turns out to be subinvariant for the evolution operator $G_{\mathcal I}(t,s)$ (${\mathcal I}\in \{{\mathcal D}, {\mathcal N}\}$) associated to the operator $\A(t)$. This fact follows from observing that
these two latter evolution operators are controlled from above by $G_{\mathcal N}(t,s)$ on the set of all the nonnegative functions $f\in C_b(\overline{\Rd_+})$.}
\end{rmk}

The following proposition and theorem deal with the asymptotic behaviour of the Dirichlet and Neumann evolution operators.
\begin{prop}
\label{prop-4.9}
Suppose that
\begin{eqnarray*}
\sup_{(t,x)\in I\times\Rd_+}\left [r(t,x)+\left (\frac{k_1^2d^2}{4}-1\right )\eta(t,x)\right ]:=2\sigma_0<0,
\end{eqnarray*}
and fix $s\in I$. Then, the following properties are satisfied.
\begin{enumerate}[\rm (i)]
\item
For any $f\in C_b(\overline{\Rd_+})$, $G_{\mathcal N}(t,s)f$ tends to $m_s^{\mathcal N}(f)$ locally uniformly in $\Rd_+$ as $t\to +\infty$.
More precisely, for any $K>0$ there exists a positive constant $c_{K,s}$ such that
\begin{equation}
\hskip 1truecm
|(G_{\mathcal N}(t,s)f)(x)-m_s^{\mathcal N}(f)|\le
c_{K,s}e^{\sigma_0(t-s)}\|f\|_{\infty},\qquad\;\,t>s,\;\,x\in B_K^+.
\label{decad-loc-unif-1}
\end{equation}
As a byproduct, $\|G_{\mathcal N}(t,s)f-m_s^{\mathcal N}(f)\|_{L^p(\Rd_+,\mu_t^{\mathcal N})}$ tends to $0$ as $t\to +\infty$ for any $f\in L^p(\Rd_+,\mu_s^{\mathcal N})$ and any
$p\in [1,+\infty)$.
\item
For any $f\in C_b(\Rd_+)$, $G_{\mathcal D}(t,s)f$ tends to $0$ as $t\to +\infty$, locally uniformly in $\Rd_+$.
More precisely, for any $K>0$, there exists a positive constant $c'_{K,s}$ such that
\begin{eqnarray*}
\hskip 1truecm
|(G_{\mathcal D}(t,s)f)(x)|\le
c'_{K,s}e^{\sigma_0(t-s)}\|f\|_{\infty},\qquad\;\,t>s,\;\,x\in B_K^+.
\end{eqnarray*}
As a byproduct, $\|G_{\mathcal D}(t,s)f\|_{L^p(\Rd_+,\mu_t^{\mathcal N})}$ tends to $0$ as $t\to +\infty$ for any $f\in L^p(\Rd_+,\mu_s^{\mathcal N})$ and any
$p\in [1,+\infty)$.
\item
The evolution system of measures $\{\mu_t^{\mathcal N}\}_{t\in I}$ is the unique tight evolution family associated with the operator $G_{\mathcal N}(t,s)f$,
and, for any $t\in I$, $\tilde\mu_t^{\varepsilon}$ weakly$^*$ converges to $\mu_t^{\mathcal N}$ as $\varepsilon\to 0^+$.
\end{enumerate}
\end{prop}

\begin{proof}
(i) We fix $f\in C^1_b(\overline{\Rd_+})$, and observe that
\begin{eqnarray*}
(G_{\mathcal N}(t,s)f)(x)\hskip -.5mm - \hskip -.5mm m_s^{\mathcal N}(f)\hskip -.5mm =\hskip -.5mm \langle \mu_t^{\mathcal N},(G_{\mathcal N}(t,s)f)(x)\hskip -.5mm -\hskip -.5mm G_{\mathcal N}(t,s)f\rangle,\quad\;\,t>s\in I,\,x\in\Rd_+.
\end{eqnarray*}
We now set $r(t)=e^{-\sigma_0(t-s)/2}$, $A_t=\R^d_+\setminus B_{r(t)}$.
Thanks to \eqref{Hsp} and Chebyshev inequality, we can estimate
$\mu_t^{\mathcal N}(A_t)\le H_{s,1}e^{\sigma_0(t-s)}$ for any $t>s$. Moreover,
\begin{eqnarray*}
\int_{\Rd_+}|y|d\mu_t^{\mathcal N}\le \int_{\Rd_+}\varphi(y)d\mu_t^{\mathcal N}\le H_{s,1},\qquad\;\,t>s,
\end{eqnarray*}
where $H_{s,1}$ is given by \eqref{Hs}.
Therefore, using Corollary \ref{coro_uniform}, where we can take $C_2=2\sigma_0$, we deduce that
\begin{align*}
& |(G_{\mathcal N}(t,s)f)(x)-m_s^{\mathcal N}(f)|\notag\\
\le &
\int_{A_t} |(G_{\mathcal N}(t,s)f)(x)-G_{\mathcal N}(t,s)f|d\mu_t^{\mathcal N}
+ \int_{B_{r(t)}} |(G_{\mathcal N}(t,s)f)(x)-G_{\mathcal N}(t,s)f|d\mu_t^{\mathcal N}\nonumber\\
\leq &
2 \|f\|_{\infty}\mu_t^{\mathcal N}(A_t)+ \|\nabla_x G_{\mathcal N}(t,s)f\|_\infty \int_{B_{r(t)}} |x-y|d\mu_t^{\mathcal N}\nonumber\\
\leq & 2H_{s,1}e^{\sigma_0(t-s)} \|f\|_\infty +  2e^{\sigma_0(t-s)}\|f\|_{C^1_b(\overline{\Rd_+})}\left (|x|+ H_{s,1}\right ),
\end{align*}
for any $x\in\Rd$ and any $t>s$.
Now, let $f$ be a general function in $C_b(\overline{\Rd_+})$.
Splitting $G_{\mathcal N}(t,s)f=G_{\mathcal N}(t,s+1)G_{\mathcal N}(s+1,s)f$ for any $t>s+1$, and observing that, by Theorem \ref{thm_3.11}, $G_{\mathcal N}(s+1,s)f\in C^1_b(\overline{\Rd_+})$
and $m_{s+1}^{\mathcal N}(G_{\mathcal N}(s+1,s)f)=m_s^{\mathcal N}(f)$, we
get
\begin{align}\label{modena-2}
|(G_{\mathcal N}(t,s)f)(x)-m_s^{\mathcal N}(f)|\le
K_se^{\sigma_0(t-s)}\left (|x|+ H_{s,1}+1\right )\|f\|_{\infty},
\end{align}
for any $t>s$ and some positive constant $K_s$, which yields \eqref{decad-loc-unif-1}.
Raising both the sides of \eqref{modena-2} to the power $p$ and, then, integrating in $\Rd_+$ with respect to the measure $\mu_t^{\mathcal N}$, we deduce
that
\begin{align}\label{modena-3}
\|G_{\mathcal N}(t,s)f-m_s^{\mathcal N}(f)\|_{L^p(\Rd_+,\mu_t^{\mathcal N})}^p\le
K_{s,p}e^{\sigma_0p(t-s)}\|f\|_{\infty}^p,
\end{align}
for any $t>s$, any $f\in C_b(\overline{\Rd_+})$ and some positive constant $K_{s,p}$ where Remark \ref{rem-4.7} is taken into account.
Since $C_b(\overline{\Rd_+})$ is dense in $L^p(\Rd_+,\mu_s^{\mathcal N})$ for any $s\in I$, from \eqref{modena-3} we deduce that
$\|G_{\mathcal N}(t,s)f-m_s^{\mathcal N}(f)\|_{L^p(\Rd,\mu_t^{\mathcal N})}$ tends to $0$ as $t\to +\infty$, for any
$f\in L^p(\Rd,\mu_s^{\mathcal N})$.

(ii) The proof is similar to the above one, and even simpler. Indeed, from the mean value theorem and Corollary \ref{coro_uniform}  we deduce that
\begin{align*}
|(G_{\mathcal D}(t,s)f)(x)|=|(G_{\mathcal D}(t,s)f)(x)-(G_{\mathcal D}(t,s)f)(0)|
\le 2e^{\sigma_0(t-s)}|x|\|f\|_{C^1_b(\overline{\Rd_+})},
\end{align*}
for any $f\in C^1_{\mathcal D}(\Rd_+)$, for any $x\in\Rd_+$. Now, the proof follows the same lines as the proof of property (i). Hence, the details are omitted.
\medskip

(iii)
We observe that the tools used to get \eqref{decad-loc-unif-1} are the gradient estimate in Corollary \ref{coro_uniform}
and the tightness of the family of measures $\{\mu_t^{\mathcal N}\}_{t\in I}$. Hence, if $\{\mu_t\}_{t\in I}$ is another tight evolution
system of measures for the operator $G_{\mathcal N}(t,s)$, then, for any $f\in C_b(\overline{\Rd_+})$,
$G_{\mathcal N}(t,s)f$ converges to the average of $f$ with respect to the measure $\mu_s$, as $t\to +\infty$.
It thus follows that $\langle\mu_t^{\mathcal N},f\rangle=\langle\mu_t,f\rangle$ for any $t\in I$ and any $f\in C_b(\overline{\Rd_+})$, i.e., $\mu_t^{\mathcal N}
=\mu_t$ for any $t\in I$.

The arguments in the proof of Lemma \ref{lemma_4.4} now show that, for any $f\in C_b(\overline{\Rd_+})$ and any infinitesimal sequence $(\varepsilon_n)$, there exists a subsequence $(\varepsilon_{n_k})$ such that $\langle \tilde\mu_t^{\varepsilon_{n_k}},f\rangle$ tends to $\langle \mu_t^{\mathcal N},f\rangle$ as $k\to +\infty$. This implies that $\langle \tilde\mu_t^{\varepsilon},f\rangle$ converges to $\langle \mu_t^{\mathcal N},f\rangle$ as $\varepsilon\to 0^+$.
\end{proof}

The previous proposition does not provide any information on the decay rate of $\|G_{\mathcal N}(t,s)f-m_s^{\mathcal N}(f)\|_{L^p(\Rd_+,\mu_t^{\mathcal N})}$
and $\|G_{\mathcal D}(t,s)f\|_{L^p(\Rd_+,\mu_t^{\mathcal N})}$ to zero as $t\to +\infty$. However Lemma \ref{lemma_4.4} is the key tool to prove that any estimate satisfied by $G^{\varepsilon}(t,s)$ in the $L^p(\Rd,\mu_s^{\varepsilon})$--$L^p(\Rd,\mu_t^{\varepsilon})$ norm, which is uniform with respect to $\varepsilon>0$, can be extended to $G_{\mathcal D}(t,s)$ and $G_{\mathcal N}(t,s)$ in the $L^p(\Rd_+,\mu_s^{\mathcal N})$--$L^p(\Rd_+,\mu_t^{\mathcal N})$ norm. Therefore, we are able to give a more precise information about the decay rate of the previous norms assuming that
the diffusion coefficients are independent of $x$, as the following theorem shows.

\begin{thm}
\label{thm-asym}
Suppose that the diffusion coefficients are independent of $x$. Then, for any $p\in[1,+\infty)$ and $s\in I$ there exists a positive constant
$k_{p,s}$ such that
\begin{align}
&\|G_{\mathcal N}(t,s)f-m_s^{\mathcal N}(f)\|_{L^p(\Rd_+,\mu_t^{\mathcal N})}\leq k_{p,s} e^{-L_0\eta_0(t-s)}\|f\|_{L^p(\Rd_+,\mu_s^{\mathcal N})},
\label{tisana}\\
&\|G_{\mathcal D}(t,s)f\|_{L^p(\Rd_+,\mu_t^{\mathcal N})}\leq k_{p,s} e^{-L_0\eta_0(t-s)}\|f\|_{L^p(\Rd_+,\mu_s^{\mathcal N})},
\label{tisana1}
\end{align}
for any $t>s$.
If the diffusion  coefficients and $b_i(\cdot,0)$ $(i=1,\ldots,d)$ are bounded in the whole of $I$, then the constant in \eqref{tisana1} can be taken independent of $s$.
\end{thm}

\begin{proof}
To prove estimates \eqref{tisana} and \eqref{tisana1}, we observe that \cite[Cor. 5.4]{AngLorLun} shows that for every $p>1$ there exists a constant $k_{p,s}>0$ (depending on $p$, $\|q_{ij}^\varepsilon\|_{\infty}$, $L_0$ and $\eta_0$) such that
\begin{equation}\label{scarica}
\|G^{\varepsilon}(t,s)g-m_s^{\varepsilon}(g)\|_{L^p(\Rd,\mu_t^\varepsilon)}\le k_{p,s} e^{-L_0\eta_0(t-s)}\|g\|_{L^p(\Rd,\mu_s^{\varepsilon})},
\end{equation}
for any $t>s\in I$ and $g\in L^p(\Rd,\mu_s^{\varepsilon})$, where $m_s^{\varepsilon}(g)$ denotes the average of $g$ with respect to the measure $\mu_s^{\varepsilon}$. Actually, in \cite{AngLorLun} the case $I=\R$ is considered but the same arguments can be applied in our situation and lead to \eqref{scarica} with
a constant which depends on $s$, and it is independent of $s$ if the diffusion coefficients and $b_i(\cdot,0)$ ($i=1,\ldots,d$) are bounded in the whole of $I$.

We fix $f\in C^{\infty}_c(\Rd_+)$ and write \eqref{scarica} with $g={\mathcal E}f$ and $g={\mathcal O}f$, respectively. Taking \eqref{Dir_eps} and \eqref{Neu_eps} into account and observing that
$m_s^{\varepsilon}({\mathcal E}f)=\langle\tilde\mu_s^{\varepsilon},f\rangle$
and $m_s^{\varepsilon}({\mathcal O}f)=0$, see Lemmas \ref{lemma_complete} and \ref{lemma_4.4}, we get
\begin{align*}
&\|G_{\mathcal N}^{\varepsilon}(t,s)f-m_s^{\varepsilon}({\mathcal E}f)\|_{L^p(\Rd_+,\tilde\mu_t^{\varepsilon})}\leq k_{p,s} e^{-L_0\eta_0(t-s)}\|f\|_{L^p(\Rd_+,\tilde\mu_s^{\varepsilon})},\\
&\|G_{\mathcal D}^{\varepsilon}(t,s)f\|_{L^p(\Rd_+,\tilde\mu_t^{\varepsilon})}\leq k_{p,s} e^{-L_0\eta_0(t-s)}\|f\|_{L^p(\Rd_+,\tilde\mu_s^{\varepsilon})}.
\end{align*}
Letting $\varepsilon\to 0^+$  from Step 2 in the proof of Theorem \ref{existence} and Proposition \ref{prop-4.9}(iii) we get \eqref{tisana} and \eqref{tisana1} for such a function $f$.
A straightforward density argument allows us to extend the previous estimate to any $f\in L^p(\Rd_+,\mu_s^{\mathcal N})$.
\end{proof}

Again, using Proposition \ref{prop-4.9}(iii), we conclude this section extending some results proved in \cite{AngLorLun} to this setting. The following theorem establishes the occurrence of some logarithmic Sobolev inequalities with respect to the tight evolution system of measures $\{\mu_s^\mathcal{N}\}$ and some remarkable properties of the Dirichlet and Neumann evolution operators such as hypercontractivity.

\begin{thm}\label{thm-other-prop}
Assume that the diffusion coefficients $q_{ij}$ are independent of $x$. Then the following properties hold true:
\begin{enumerate}[\rm (i)]
\item
for any $f\in C_c^1(\Rd_+)$, any $p\in (1,+\infty)$ and $s \in I$,
\begin{equation}\label{log-Sob_N}
\;\;\langle \mu_s^{\mathcal{N}},|f|^p\log |f|\rangle \le \frac{1}{p}\langle \mu_s^{\mathcal{N}},|f|^p\rangle\log(\langle \mu_s^{\mathcal{N}},|f|^p\rangle)
+\frac{p \Lambda}{2L_0\eta_0}\langle \mu_s^{\mathcal{N}}, |f|^{p-2}|\nabla f|^2\chi_{\{f\neq 0\}}\rangle,
\end{equation}
where $\Lambda:=\sup\{\langle Q(t)\xi,\xi\rangle:\,t\in I,\xi\in \Rd, |\xi|=1\}$;
\item
for any $s\in I$ and $p\in[2,+\infty)$, the Sobolev space $W^{1,p}(\Rd_+,\mu_s^{\mathcal{N}})$ is compactly embedded in $L^p(\Rd_+,\mu_s^{\mathcal{N}})$. As a consequence, for any $t>s\in I$ and $p\in (1,+\infty)$, $G_{\mathcal{D}}(t,s)$ and $G_{\mathcal{N}}(t,s)$ are compact from $L^p(\Rd_+,\mu_s^{\mathcal{N}})$ into $L^p(\Rd_+,\mu_t^{\mathcal{N}})$;
\item
for any $p,q\in (1,+\infty)$ with $p\leq e^{2L_0\eta_0^2 \Lambda^{-1}(t-s)}(q-1)+1$, $G_{\mathcal{D}}(t,s)$ and $G_{\mathcal{N}}(t,s)$ map $L^q(\Rd_+,\mu_s^{\mathcal{N}})$ to $L^{p}(\Rd_+, \mu_t^{\mathcal{N}})$
for every $t>s$ and, for $\mathcal{J}\in \{\mathcal{D},\mathcal{N}\}$,
\begin{eqnarray*}
\|G_{\mathcal{J}}(t,s)f\|_{L^{p}(\Rd_+,\mu_t^{\mathcal{N}})}\leq \|f\|_{L^q(\Rd_+,\mu_s^{\mathcal{N}})}, \qquad\;\, t>s, \;\, f \in L^q(\Rd_+,\mu_s^{\mathcal{N}}).
\end{eqnarray*}
\end{enumerate}
\end{thm}
\begin{proof}
(i) From \cite[Thm. 3.3]{AngLorLun} we know that for any $g\in C_b^1(\Rd)$, any $p\in (1,+\infty)$ and any $s \in I$,
\begin{equation}\label{fame}
\;\;\langle \mu_s^{\varepsilon},|g|^p\log |g|\rangle \le \frac{1}{p}\langle \mu_s^{\varepsilon},|g|^p\rangle\log(\langle \mu_s^{\varepsilon},|g|^p\rangle)
+\frac{p \Lambda}{2L_0\eta_0}\langle \mu_s^{\varepsilon}, |g|^{p-2}|\nabla g|^2\chi_{\{f\neq 0\}}\rangle,
\end{equation}
since $\langle Q^\varepsilon(s)\xi,\xi\rangle\le \Lambda |\xi|^2$ for any $s\in I$ and any $\xi \in \Rd$.

Now, let $f\in C^1_c(\Rd_+)$, writing \eqref{fame} with $g$ being replaced with $\mathcal E f$ and using the symmetry of $\rho^\varepsilon(t,\cdot)$ with respect the last variable, we get
\begin{equation*}
\;\;\langle \tilde\mu_s^{\varepsilon},|f|^p\log |f|\rangle \le \frac{1}{p}\langle \tilde\mu_s^{\varepsilon},|f|^p\rangle\log(\langle \tilde\mu_s^{\varepsilon},|f|^p\rangle)
+\frac{p \Lambda}{2L_0\eta_0}\langle \tilde\mu_s^{\varepsilon}, |f|^{p-2}|\nabla f|^2\chi_{\{f\neq 0\}}\rangle.
\end{equation*}
Hence the claim follows applying Proposition \ref{prop-4.9}(iii).

(ii) Once \eqref{log-Sob_N} is established, the proof follows as in \cite[Thm. 3.4]{AngLorLun}.

(iii) \cite[Thm 4.1]{AngLorLun} yields that if $p,q\in (1,+\infty)$ and $p\leq e^{2L_0\eta_0^2 \Lambda^{-1}(t-s)}(q-1)+1$ then $G^\varepsilon(t,s)$ map $L^q(\Rd,\mu_s^{\varepsilon})$ to $L^{p}(\Rd, \mu_t^{\varepsilon})$
for every $t>s$ and
\begin{equation*}
\|G^\varepsilon(t,s)g\|_{L^{p}(\Rd,\mu_t^{\varepsilon})}\leq \|g\|_{L^q(\Rd,\mu_s^{\varepsilon})}, \qquad\;\, t>s, \;\, g \in L^q(\Rd,\mu_s^{\varepsilon}),
\end{equation*}
Now, arguing as in the proof of Theorem \ref{thm-asym} we get the claim.
\end{proof}

\section{Examples}
\label{sect-5}
In this section, we exhibit two classes of nonautonomous elliptic operators $\A(t)$ to which the main results of this paper apply.

Let $\A(t)$ be defined on smooth functions $\zeta$ by
\begin{align}
(\mathcal{A}(t)\zeta)(x)=&(1+|x|^2)^k{\rm Tr}(B(t,x)D^2\zeta(x))-b_0(t)(1+|x|^2)^m\langle x,\nabla
\zeta(x)\rangle\notag\\
&+g(t,x_d)D_{d}\zeta(x)-\gamma(t)(1+|x|^2)^q\zeta(x),
\label{ex-A}
\end{align}
for any $(t,x)\in I\times\R^{d}_+$, where $I$ is an open halfline (possibly $I=\R$). We assume the following conditions on the coefficients of the operator $\A(t)$.
\begin{hyp}
\label{hyp4}
\begin{enumerate}[\rm (i)]
\item
$k,m,q\in (1,+\infty)$ with $k\le m<q$;
\item
for any $i,j=1,\ldots,d$, $b_{ij}=b_{ji}$ belongs to $C^{\alpha/2,1}_{\rm loc}(I\times\overline{\Rd_+})$. Moreover, the $b_{ij}$'s and their first-order spatial derivatives are bounded in $I\times\Rd_+$, and $b_{id}(\cdot,0)\equiv 0$ $(i=1,\ldots,d-1)$;
\item
there exist a positive constant $\eta_0$ such that
$\langle B(t,x)\xi,\xi\rangle\ge\eta_0 |\xi|^2$,
for any $t\in I$, $x\in\Rd_+$ and $\xi\in\Rd$;
\item
$b_0\in C^{\alpha/2}_{\rm loc}(I)$,
$g\in C^{\alpha/2,1}_{\rm loc}(I\times\overline{\R_+})$, $g(\cdot,0)\equiv 0$, $D_dg(t,x_d)\le \vartheta(t)(1+x_d^2)^m$ for any
$(t,x)\in I\times\R_+$ and some positive function $\vartheta$ such that $b_0(t)-\vartheta(t)\ge\beta_0$ for any $t\in I$;
\item
$\gamma\in C^{\alpha/2}_{\rm loc}(I)\cap C_b(I)$ and there exists a positive constant $\gamma_0$ such that
 $\gamma(t)\ge\gamma_0$ for any $t\in I$.
\end{enumerate}
\end{hyp}
Under such assumptions, Hypotheses \ref{hyp2} are satisfied. Indeed, we can take
\begin{eqnarray*}
r(t,x)=-(b_0(t)\hskip -.5mm - \hskip -.5mm \vartheta(t))(1+|x|^2)^m,\quad\,(t,x)\in I\times\Rd_+,\qquad\;\; L_0=\beta_0\eta_0^{-1},\;\,L_1=0.
\end{eqnarray*}
Moreover, if we set $q_{ij}(t,x)=b_{ij}(t,x)(1+|x|^2)^k$, it holds that
\begin{eqnarray*}
|\nabla_xq_{ij}(t,x)|\le (\|\nabla_xb_{ij}\|_{\infty}+k\|b_{ij}\|_{\infty})(1+|x|^2)^{k},\qquad\;\,(t,x)\in I\times\Rd_+,
\end{eqnarray*}
so that we can take $k_1=\sup_{i,j\le d}(\|\nabla_xb_{ij}\|_{\infty}+k\|b_{ij}\|_{\infty})\eta_0^{-1}$.

Finally, conditions \eqref{cond-deriv-c} are satisfied with
$\beta(t,x)=2q\|\gamma\|_{\infty}(1+|x|^2)^{q-1/2}$, for any $(t,x)\in I\times\Rd$, and $k_2=2q\|\gamma\|_{\infty}\gamma_0^{-1}$.
It thus follows that
\begin{align}
&r(t,x)+\left (\frac{k_1^2d^2}{4M_p}-M_p\right )\eta(t,x)-\left (1-\frac{1}{p}\right )c(t,x)+\frac{pk_2}{4(p-1)}\beta(t,x)\notag\\
\le&-\beta_0 (1+|x|^2)^m+\bigg (\frac{\sup_{i,j\le d}(\|\nabla_xb_{ij}\|_{\infty}+k\|b_{ij}\|_{\infty})^2d^2}{4M_p\eta_0^2}-M_p\bigg )\eta_0(1+|x|^2)^k\notag\\
&-\left (1-\frac{1}{p}\right )\gamma_0(1+|x|^2)^q+\frac{pq^2}{(p-1)\gamma_0}\|\gamma\|_{\infty}^2(1+|x|^2)^{q-1/2}.
\label{esempio-1}
\end{align}
Since $q>m\ge k$, the left-hand side of \eqref{esempio-1} is bounded from above in $I\times\Rd_+$ for any $p\in (1,+\infty)$.
Theorems \ref{Dir_thm}, \ref{thm_3.11} apply and estimate \eqref{GD_GN-5} holds, with the constant $C_p$ being given by the supremum over $\Rd_+$ of the right hand side of \eqref{esempio-1}.

If $\gamma\equiv 0$ and the other conditions in Hypotheses \ref{hyp4} are satisfied, then estimates \eqref{GD_GN_0} and \eqref{nor_p_fun} hold true, respectively with
\begin{align*}
K_p&=\sup_{y\ge 1}\left (-\beta_0y^m+\frac{\sup_{i,j\le d}(\|\nabla_xb_{ij}\|_{\infty}+k\|b_{ij}\|_{\infty})^2d^2}{4\eta_0M_p}y^k\right ),\\[1.5mm]
C_p&=\sup_{y\ge 1}\left [-\beta_0y^m+\left (\frac{\sup_{i,j\le d}(\|\nabla_xb_{ij}\|_{\infty}+k\|b_{ij}\|_{\infty})^2d^2}{4\eta_0M_p}-\eta_0\right )y^k\right ],
\end{align*}
for any $p\in (1,+\infty)$.

Further, if $\sup_{i,j\le d}(\|\nabla_xb_{ij}\|_{\infty}+k\|b_{ij}\|_{\infty})^2d^2<4\eta_0^2(\eta_0+\beta_0)$,
then, the results in Proposition \ref{prop-4.9} hold true with
\begin{eqnarray*}
\sigma_0=-\beta_0+\frac{\sup_{i,j\le d}(\|\nabla_xb_{ij}\|_{\infty}+k\|b_{ij}\|_{\infty})^2d^2}{4\eta_0^2}-\eta_0.
\end{eqnarray*}

Finally, we consider the case when the diffusion coefficients of the operators $\A(t)$ in \eqref{ex-A} are independent of $x$, and the following conditions are satisfied.
\begin{hyp}
\label{hyp5}
\begin{enumerate}[\rm (i)]
\item
$m,q\in\N$ with $2q-1\le m$;
\item
$b_{ij}=b_{ji}$ belongs to $C^{\alpha/2}_{\rm loc}(I)$ for any $i,j=1,\ldots,d$;
\item
$b_{id}(t)=0$ for any $t\in I$ and $i=1,\ldots,d-1$;
\item
there exist a positive constant $\eta_0$ such that
$\langle B(t)\xi,\xi\rangle\ge\eta_0 |\xi|^2$,
for any $t\in I$ and $\xi\in\Rd$;
\item
Hypotheses $\ref{hyp4}${\rm (iv)}-{\rm (v)} are satisfied.
\end{enumerate}
\end{hyp}
In this case estimate \eqref{GD_GN} is satisfied with $p=1$ and $C_1=4q^2\|\gamma\|_{\infty}^2/\beta_0$.

If $\gamma\equiv 0$ and the other conditions in Hypotheses \ref{hyp5} are satisfied, then \eqref{GD_GN_0}, \eqref{tisana} and \eqref{tisana1} hold true and we can take $L_0=\beta_0\eta_0^{-1}$.

\appendix

\section{An auxiliary result}

Here, we prove a result which is used in the proof of Theorem \ref{GD_GN} and provides us with a (local) higher spatial H\"older regularity of the third-order derivatives of the solution of the Cauchy-Neumann problem
\begin{equation}\label{auxiliary}
\left\{
\begin{array}{ll}
u_t(t,x)=\A(t)u(t,x), &t\in (0,T),\,x\in \Om\\[1mm]
\displaystyle{\frac{\partial u}{\partial \nu}(t,x)}=0,&t\in (0,T),\,x\in \partial \Om,\\[2mm]
u(s,x)= f(x),& x\in \overline{\Om},
\end{array}
\right.
\end{equation}
in a bounded domain $\Omega$ of class $C^{2+\alpha}$ (for some $\alpha\in (0,1)$),
without assuming any H\"older regularity in $t$ of the spatial gradient of
the coefficients of the uniformly elliptic operator $\A(t)$, defined on smooth functions $\psi$ by
\begin{eqnarray*}
(\mathcal{A}(t)\psi)(x)= 
{\rm Tr}(Q(t,x)D^2\psi(x))+\langle b(t,x),\rangle \nabla\psi(x)\rangle-c(t,x)\psi(x),
\end{eqnarray*}
for any $t\in [0,T]$ and any $x\in\Omega$.
Even if it seems quite predictable we did not find any reference for this result. Hence, for the sake of completeness we provide a proof of it.

\begin{prop}\label{smoothdatum}
Assume that the coefficients of the operator $\A(t)$ belong to $C^{\alpha/2,\alpha}((0,T)\times\Om)\cap C^{0,1+\alpha}([0,T]\times\Om)$. Then, for any $f\in C^{3+\alpha}(\Om)$
with normal derivative identically vanishing on $\partial\Omega$, problem \eqref{auxiliary} has a unique solution $u \in C^{1+\alpha/2,2+\alpha}((0,T)\times \Om)\cap C^{0,3+\alpha}_{\rm loc}([0,T]\times\Omega)$.
\end{prop}
\begin{proof}
By \cite[Thm IV.5.3]{LadSolUra68Lin}, problem \eqref{auxiliary} admits a unique solution $u$ which belongs to $C^{1+\alpha/2,2+\alpha}((0,T)\times\Omega)$ and
there exists a positive constant $C_1$ such that
\begin{equation}
\|u\|_{C^{1+\alpha/2,2+\alpha}([0,T]\times\Omega)}\le C_1\|f\|_{C^{2+\alpha}(\Omega)}.
\label{stima-lady}
\end{equation}
Moreover, we claim that $D_ju$ is $(1+\alpha)/2$-H\"older continuous in $t$, uniformly with respect to $x\in\Rd$, for any $j=1,\ldots,d$. Indeed, writing
\begin{eqnarray*}
u(t,x)-u(s,x)=\int_s^t u_t(\sigma,x)d\sigma,\qquad\;\, t, s\in[s,T],\, x\in \overline{\Om},
\end{eqnarray*}
we can easily show that
$\|u(t,\cdot)-u(s,\cdot)\|_{C^{\alpha}(\Om)}\le \|u_t\|_{C^{0,{\alpha}}([0,T]\times \overline{\Om})}|t-s|$ for any $t,s\in [0,T]$.
Since $C^1(\overline\Omega)$ belongs to the class $J_{(1-\alpha)/2}$ between $C_b^{\alpha}(\Omega)$ and $C^{2+\alpha}_b(\Omega)$, there exists a constant $K$, independent of $s,t$ such that
\begin{align*}
\|u(t,\cdot)-u(s,\cdot)\|_{C^1(\overline{\Om})}\le & K\|u(t,\cdot)-u(s,\cdot)\|_{C^{\alpha}(\Omega)}^{\frac{1+\alpha}{2}}\|u(t,\cdot)-u(s,\cdot)\|_{C^{2+\alpha}(\Omega)}^{\frac{1-\alpha}{2}}\\
\le &2^{\frac{1-\alpha}{2}}KC_1\|f\|_{C^{2+\alpha}(\Omega)}|t-s|^{\frac{1+\alpha}{2}},
\end{align*}
for any $s,t\in [0,T]$, so that the claim follows and
\begin{equation}
\|D_ju\|_{C^{(1+\alpha)/2,0}([0,T]\times\Omega)}\le 2^{(1-\alpha)/2}KC_1\|f\|_{C^{2+\alpha}(\Omega)},\qquad\;\, j=1,\ldots,d.
\label{stima-holder-grad}
\end{equation}

To prove that $u$ admits third-order spatial derivatives in $C^{0,\alpha}_{\rm loc}([0,T]\times \Omega)$, we fix an open set $\Omega'$ compactly contained in $\Omega$, and a function
$\vartheta\in C^{\infty}(\Rd)$ such that $\chi_{\Omega'}\le\vartheta\le\chi_{\Omega''}$ for some open set $\Omega''$ compactly contained in $\Omega$.
Let $v\in C^{1+\alpha/2,2+\alpha}_b([0,T]\times\Rd)$ denote the trivial extension of the function $u\vartheta$ to the whole of $\Rd$. As it is The function $v$ solves the Cauchy problem
\begin{eqnarray*}
\left\{
\begin{array}{ll}
v_t(t,x)=\hat\A(t)v(t,x)+\hat g(t,x), &t\in (0,T),\,x\in \Rd,\\[1mm]
v(s,x)= \vartheta(x)f(x),& x\in\Rd,\\[1mm]
\end{array}
\right.
\end{eqnarray*}
where, with a slight abuse of notation, we still denote by $\vartheta f$ the trivial extension of this function to the whole of $\Rd$. Here,
$\hat\A(t)={\rm Tr}(\hat QD^2)+\langle\hat b,\nabla \rangle-\hat c$, $\hat Q=\eta Q+(1-\eta)I$, $\hat b=\eta b$ and $\hat c=\eta c$, where
$\eta\in C^{\infty}_c(\Omega)$ is a smooth function satisfying $\eta\equiv 1$ in $\Omega''$. Finally,
$\hat g$ denotes the trivial extension to the whole of $[0,T]\times\Rd$ of the function
$-u(\hat\A(t)+\hat c)\vartheta-2\langle \hat Q\nabla\vartheta,\nabla_xu\rangle$.

For any $j\in\{1,\ldots,d\}$, any $h>0$ and any $\psi:[0,T]\times\Rd\to\R$, we denote by $\tau_h\psi$ the function defined by $\tau_h\psi=h^{-1}(\psi(\cdot,\cdot+he_j)-\psi)$. Clearly, $v_h:=\tau_hv$ belongs
to $C^{1+\alpha/2,2+\alpha}_b([0,T]\times\Rd)$ and solves the Cauchy problem
\begin{eqnarray*}
\left\{
\begin{array}{ll}
D_tv_h(t,x)=\hat\A(t)v_h(t,x)+\tau_h\hat g(t,x)+F_h(t,x), &t\in (0,T),\,x\in \Rd,\\[1.5mm]
v_h(s,x)= \tau_h(\vartheta f)(x),& x\in\Rd,
\end{array}
\right.
\end{eqnarray*}
where
\begin{eqnarray*}
F_h=\sum_{i,j=1}^d(\tau_h\hat q_{ij})D_{ij}v(\cdot,\cdot+he_j)+\sum_{j=1}^d(\tau_h\hat b_j)D_jv(\cdot,\cdot+he_j)-(\tau_h\tilde c) v(\cdot,\cdot+he_j).
\end{eqnarray*}
By the results in \cite{KCL, KCL-1, KCL-2} the $C^{0,2+\alpha}_b([0,T]\times\Rd)$-norm of $v_h$ can be estimated from above
by a positive constant, independent of $h$, times the sum of the $C^{2+\alpha}_b(\Rd)$-norm of $\tau_h(\vartheta f)$ and the $C^{0,\alpha}_b([0,T]\times\Rd)$-norm
of $\tau_h\tilde g$ and $F_h$. More precisely, taking \eqref{stima-lady} into account, we can write
\begin{align*}
\|v_h\|_{C^{0,2+\alpha}_b}\le C_2\bigg (&\|\tau_h(\vartheta f)\|_{C^{2+\alpha}_b}+\|\tau_h\tilde g\|_{C^{0,\alpha}_b}+\|f\|_{C^{2+\alpha}_b}\|\tau_h\tilde c\|_{C^{0,\alpha}_b}
\\
&+\|f\|_{C^{2+\alpha}_b}\sum_{i,j=1}^d\|\tau_h\tilde q_{ij}\|_{C^{0,\alpha}_b}+\|f\|_{C^{2+\alpha}_b}\sum_{j=1}^d\|\tau_h\tilde b_j\|_{C^{0,\alpha}_b}
\bigg ),
\end{align*}
with $C_2$ (as all the forthcoming constants) being independent of $h$ and $f$, and where $C^{2+\alpha}_b$ and $C^{0,\alpha}_b$ stand for  $C^{2+\alpha}_b(\Rd)$ and $C^{0,\alpha}_b([0,T]\times\Rd)$, respectively.

Since
\begin{equation}
\tau_h\psi(x)=\int_0^1D_j\psi(x+hse_j)ds,\qquad\;\,x\in\Rd,
\label{form-integrale}
\end{equation}
for any $\psi\in C^1(\Rd)$, it follows that
\begin{align*}
\|v_h\|_{C^{0,2+\alpha}_b}\le C_2\bigg (&\|\vartheta f\|_{C^{3+\alpha}_b}+\|\tilde g\|_{C^{0,1+\alpha}_b}+\|f\|_{C^{2+\alpha}_b}\|\tilde c\|_{C^{0,1+\alpha}_b}
\\
&\;+\|f\|_{C^{2+\alpha}_b}\sum_{i,j=1}^d\|\tilde q_{ij}\|_{C^{0,1+\alpha}_b}+\|f\|_{C^{2+\alpha}_b}\sum_{j=1}^d\|\tilde b_j\|_{C^{0,1+\alpha}_b}
\bigg ).
\end{align*}
The definition of $\tilde g$, estimate \eqref{stima-lady}
and the assumptions on $f$ and the coefficients of the operator $\A(t)$ allow us to conclude that
\begin{equation}
\|v_h\|_{C^{0,2+\alpha}_b}\le C_3\|f\|_{C^{3+\alpha}(\Omega)}.
\label{stima-2+alpha}
\end{equation}

Further, taking advantage of \eqref{stima-holder-grad} and \eqref{form-integrale}, we can easily check that
$v_h$ belongs to $C^{(1+\alpha)/2,0}_b([0,T]\times\Rd)$ and $\|v_h\|_{C^{(1+\alpha)/2,0}_b([0,T]\times\Rd)}$ can be estimated by a constant independent of $h$.
Again, since $C_b^2(\Rd)$ is of class $J_{2/(2+\alpha)}$ between $C_b(\Rd)$ and $C^{2+\alpha}_b(\Rd)$, from \eqref{stima-2+alpha} we deduce that
\begin{equation}
\|v_h(t,\cdot)-v_h(s,\cdot)\|_{C^2_b(\Rd)}\le C_4|t-s|^{\theta},\qquad\;\, s,t\in [0,T],
\label{stima-tempo}
\end{equation}
where $\theta=\alpha(1+\alpha)/(4+2\alpha)$. Combining \eqref{stima-2+alpha} and \eqref{stima-tempo} it follows that
$v_h\in C^{\theta/2,2+\theta}_b([0,T]\times\Rd)$ and
$\|v_h\|_{C^{\theta/2,2+\theta}_b([0,T]\times\Rd)}\le C_5$.
Now, we can use a compactness argument and conclude that $D_jv$ belongs to $C^{0,2+\alpha}_b([0,T]\times\Rd)$.
Recalling that $\vartheta\equiv 1$ in $\Omega'$, we deduce that
$u\in C^{0,3+\alpha}([0,T]\times\Omega')$.
\end{proof}

\end{document}